\documentclass[11pt,oneside,english,final]{amsart}

\usepackage[utf8]{inputenc} 
\usepackage[T1]{fontenc} 
\usepackage{amssymb,amsthm,amsmath} 

\usepackage[unicode=true,bookmarks=false, breaklinks=false,pdfborder={0 0 1},backref=false,colorlinks=false]{hyperref}

\usepackage[left=2.5cm,top=3.5cm,right=2.5cm,bottom=4.5cm,nohead]{geometry}
\usepackage[notref,notcite]{showkeys}

\usepackage[usenames]{color}
\definecolor{Red}{rgb}{0.9,0.1,0.3}

\newcommand{\ev}[1]{\mathbb{E}{#1}}
\newcommand{\e}[1]{\mathbb{E}}

\newcommand{\norm}[3]{\Vert #1 \Vert_{ #2 } ^{ #3 }}
\newcommand{\rbr}[1]{\left( #1 \right)}
\newcommand{\sbr}[1]{\left[ #1 \right]}
\newcommand{\cbr}[1]{\left\{ #1 \right\}}
\newcommand{\ddp}[2]{\left\langle #1, #2 \right\rangle}
\newcommand{\intr}{\int_{\mathbb{R}^d}}
\newcommand{\inti}{\int_{0}^{+\infty}}
\newcommand{\intc}[1]{\int_{0}^{#1}}
\newcommand{\Rd}{\mathbb{R}^d}

\newcommand{\R}{\mathbb{R}}
\newcommand{\T}[1]{\mathcal{T}_{#1}}

\newcommand{\var}{\text{Var}}

\definecolor{czerwony}{rgb}{1,0.3,0.3}
\definecolor{zielony}{rgb}{0.1,0.8,0.4}
\newcommand{\dd}[1]{\textnormal{d}#1}

\newcommand{\SP}{\mathcal{S}'(\Rd)}
\newcommand{\SD}{\mathcal{S}(\Rd)}

\newcommand{\lap}[1]{\ev{\exp\cbr{#1}}}
\newcommand{\ceq}{\eqsim}
\newcommand{\cleq}{\lesssim}

\theoremstyle{plain} 

\newtheorem{theorem}{Theorem}[section]
\newtheorem{fact}{Fact}[section]
\newtheorem{lem}{Lemma}[section]
\newtheorem{lemma}{Lemma}[section]

\theoremstyle{remark}\newtheorem{rem}{Remark}[section]
\newtheorem*{acknowledgement*}{Acknowledgement}

\title[]{Fluctuations of the occupation times for branching system starting from infinitely divisible point processes.}
\author{Piotr Mi\l{}o\'s \\{\tiny Faculty of Mathematics, Informatics, and Mechanics,\\Banacha 2, 02-097 Warszawa, Poland\\email:pmilos@mimuw.edu.pl}}
\begin{document}
	
	\maketitle
	\begin{abstract}
		
	In the paper the rescaled occupation time fluctuation process of a certain empirical system is investigated. The system consists of particles evolving independently according to a symmetric $\alpha$-stable motion in $\Rd$, $\alpha<d<2\alpha$. The particles split according to the binary critical branching law with intensity $V>0$. We study how the \emph{limit} behaviour of the fluctuations of the occupation time depends on the \emph{initial particle configuration}. We obtain a functional central limit theorem for a vast class of infinitely divisible distributions. Our findings extend and put in a unified setting results of \cite{Bojdecki:2006ab} and \cite{Milos:2007ab}, which previously seemed to be disconnected. The limit processes form a one dimensional family of long-range dependance centred Gaussian processes. 
	\end{abstract}

	AMS subject classification: primary 60F17, 60G20, secondary 60G15\\

	Key words: Functional central limit theorem; Occupation time fluctuations; Branching
	particles systems; Generalised Wiener process, Infinitely divisible random measure, Point processes
	\pagestyle{empty}
\section{Introduction} % (fold)
\label{sec:aim}
We consider a system of particles in $\mathbb{R}^{d}$ starting off at time $t=0$ from a certain distribution (to be described later). They evolve independently, moving according to a symmetric $\alpha$-stable L\'evy process and undergoing the critical binary branching (i.e., 0 or 2 particles with probability $\frac{1}{2}$ in each case) at a rate $V>0$. We concentrate on the case when $\alpha<d<2\alpha$ and the starting measure is an \emph{infinitely divisible point process} $M$. We denote this system by $N$. It can be identified with the empirical process $\cbr{N_t}_{t\geq 0}$ which is measure-valued. For a Borel set $A$, the random variable $N_t(A)$ indicates the number of particles of the system in set $A$ at time $t$. We define the rescaled process of the fluctuations of the occupation time $\cbr{X_T(t)}_{t\in[0,1]}$ by:
\begin{equation}
 X_T(t) := \frac{1}{F_T} \intc{Tt} (N_s-\lambda) \dd{s}, \label{eq:fluctuatons_process}
\end{equation}
where $F_T := T^{(3-\frac{d}{\alpha})/2}$. The main object of our investigation is the limit of $X_T$ as the time is accelerated, $T\rightarrow +\infty$. We are interested in how the limit depends on the starting measure $M$. In our paper we derive a simple integral criterion which determines the behaviour of the limit. The exposition will be clearer  when put against the state-of-art in the field, consequently we start by recalling the most relevant results. The reader acquainted with them can proceed to the next paragraph. Studies of the fluctuations of the occupation time of branching particle systems started with pioneering papers \cite{Cox:1984aa, Cox:1985aa}. The field has received a lot of research attention recently due to the series of papers by Bojdecki at al. \cite{Bojdecki:2006aa,Bojdecki:2006ab,Bojdecki:2007aa,Bojdecki:2007ab,Bojdecki:2007ac,Bojdecki:2007ad,Bojdecki:2008aa,Bojdecki:2008xd,Bojdecki:2010bl} and latter Milos \cite{Milos:2007ab,Milos:aa,Mios:2009oq,Milos:2008aa}. An analogous problem for a system of particles on the integer lattice was studied in \cite{Birkner:2007aa}. The aforementioned papers concentrate on proving central limit theorems (usually in a functional setting) for various configurations of parameters  (i.e. various values of $d,\alpha$, branching laws and starting distributions). These findings were also complemented by corresponding large and moderate deviation principles e.g. \cite{Deuschel:1994aa,Deuschel:1998aa, Hong:2005aa,Mios:2009jl}. For a system with a critical finite variance branching law (e.g. the critical binary law) one can distinguish three main regimes of the limit behaviour:
\begin{itemize}
	\item low dimensions ($d<\alpha$), when the system suffers from local extinction and direct study of the fluctuations of the rescaled occupation time does not make sense,
	\item intermediate dimensions ($\alpha<d<2\alpha$), when the limit has a simple spatial structure (Lebesgue measure) and a complicated temporal one (with long-range dependence property),
	\item  large dimensions ($2\alpha<d$), when the limit has a complicated spatial structure ($\SP$-valued random field) and a simple temporal one (process with independent increments).
\end{itemize}
The case of the \emph{intermediate dimensions}, in which we concentrate in this paper, is interesting from a number of reasons. Let us recall results of  \cite{Bojdecki:2006ab,Milos:2007ab}. They considered two starting measures, a homogenous Poisson point process $Poiss$ and the equilibrium measure $Eq$. The latter is the unique invariant distribution of the dynamics of the branching particle system. It exists under assumption $d>\alpha$ and can also be defined by (see \cite{Gorostiza:1991aa})
\begin{equation}
	N_t \rightarrow^d Eq, \quad t\rightarrow +\infty, \label{eq:defEqulibrium}
\end{equation}
where $\rightarrow^d$ denotes the convergence in law and $N$ starts off from $Poiss$. Surprisingly, the limits of $X_T$ are \emph{different} in these cases (a so-called sub-fractional Brownian motion and a fractional Brownian motion, respectively).  A natural question that arose was about the properties of starting distributions that are responsible for different forms of the limits. It became even more interesting in light of results of \cite{Bojdecki:2010bl} (to be discussed in Remark \ref{rem:bojecki}). This paper brings an answer to this question for a vast and important class of starting distributions, viz. infinitely divisible point processes. We were able to factor out a \emph{single characteristic} of a starting distribution $M$ (named $\mathcal{H}(M)$ see (A2) in Section \ref{sub:assumptions}) which determines the limit process. Precisely, its temporal structure (up to multiplicative constants) is described by a centred Gaussian process of the form
\[
	 \xi + \mathcal{H}(Q) \zeta,
\]
where $\xi$ is a sub-fractional Brownian motion, $\zeta$ is another centred Gaussian process, $\xi$ and $\zeta$ are independent. In this way we obtained a whole spectrum of limit processes. In our framework, $Poiss$ and $Eq$ starting distributions are special cases. Moreover, it is possible to give some intuitive description of $\mathcal{H}(M)$. An infinitely divisible point process is characterised by its canonical measure. Roughly speaking, this measure defines ``clans of particles'' which arrive to the system together in groups. The condition $\mathcal{H}(M)>0$ implies that on average the size of a clan is infinite and asymptotically there are $\mathcal{H}(Q) r^\alpha $ particles in the ball of radius $r$. For the equilibrium measure these ``clans'' are exactly the clans which can be described in terms of genealogy bonds among particles (e.g. in \cite{Gorostiza:1991aa, Zahle:2002aa}). We do not know an exact explanation for such behaviour, although we expect that recurrence and persistence properties (e.g. like the one described in \cite{Stoeckl:1994aa}) play a crucial role. A clan has to be large enough so that its recurrence influences the limit.  It is also worth mentioning that the limit process exhibits long-range dependence, which is also a consequence of the recurrence properties. Finally, we return to \cite{Bojdecki:2010bl}; studying an analogous problem for particles systems without branching they obtained the same family of processes. We attempt to give an intuitive explanation of our results as well as understand the relation to the ones of \cite{Bojdecki:2010bl} in Remark \ref{rem:intuitve} and Remark \ref{rem:bojecki}.\\
The paper is organised as follows. In  Section \ref{sec:preliminaries} we gather definitions and facts used in the subsequent sections. In Section \ref{sec:result} we present the main result of the paper, Theorem \ref{thm:starting}, along with examples and comments. Sections \ref{sec:proofs} and \ref{sec:calculations} are devoted to the proof of Theorem \ref{thm:starting}. Finally, in Section \ref{sec:equlibrium} we collected some basic facts about the equilibrium distribution.   
% section aim (end)

% paragraph symbols (end)

\section{Preliminaries} % (fold)
\label{sec:preliminaries}
For the reader convenience we put a short index of symbols 
\paragraph{Symbols}\label{par:symbols}~\\  % (fold)
$\mathcal{M}(\Rd)$ - space of locally finite point measures on $\Rd$ with topology of vague convergence \\
$Q_x$ - Palm measure of random measure $Q$\\
$[\mu]$ - set of atoms of a point measure $\mu$\\
$\lambda$ - Lebesgue measure\\
$\theta_x$ - shift operator\\
$\SD, \SP$ - Schwartz space of rapidly decreasing functions and its dual space of distributions\\
$\mathcal{C}([0,1], \mathcal{X})$ - the space of $\mathcal{X}$-valued continuous functions with the $\sup$ norm\\
$p_t(\cdot), \T{t}$ - density of the symmetric $\alpha$-stable motion and its semigroup\\
For a space of functions $\mathcal{X}$ by $\mathcal{X}_+$ we denote the subspace of  positive functions\\
%${M}(\Rd)$ - measurable, measurable positive functions  \\
$\ddp{\cdot}{\cdot}$ - denotes either duality or integration \\
$\ceq, \cleq$ - we use this notation when an equality or inequality holds with a constant $c>0$ which is irrelevant for calculations. E.g. $f(x)\ceq g(x)$ means that there exists a constant $c>0$ such that $f(x) = c g(x)$.
\subsection{Functional setting} % (fold)
\label{sub:functional_setting}
Process $X_T$ is a measure-valued. However, from technical reasons, it is convenient to embed it in $\SP$ i.e. the space of tempered distributions dual to the space of rapidly decreasing functions $\SD$. More precisely we will consider $X_T$ to be a $\mathcal{C}([0,1], \SP)$ random variable and prove the convergence in law with respect to the topology of this space. 
% subsection functional_setting (end)
\subsection{Gaussian processes with long-range dependance } \label{sec:gaussian}
In our paper we will need two Gaussian processes. The first one is the sub-fractional Brownian motion. It was introduced in \cite{Bojdecki:2004aa} where also its properties were studied. It is a centred Gaussian process with the covariance function
\begin{equation}
	 C_h(s,t) := s^h + t^h - \frac{1}{2}\sbr{(s+t)^h + |s-t|^h},\quad h\in(0,2).  \label{eq:subfroctional}
\end{equation}
We will also need another centred Gaussian process with covariance
\begin{equation}
	c_h(s,t) := sgn(h-1)\sbr{(s+t)^h - s^h - t^h}, \quad h\in (0,2). \label{eq:covariance}
\end{equation}
As indicated in \cite{Bojdecki:2010bl}[after (2.4)] for $h\neq 1$ function $c_h$ is positive-definite since
\[
	c_h(s,t) \ceq \intc{s}\intc{t} \int_{\R} e^{-r|x|^{2-h}}e^{-r'|x|^{2-h}} \dd{x} \dd{r'}\dd{r}.   
\]	
As the limit in Theorem \ref{thm:starting} we obtain a process which is sum of two independent processes as above, viz. it is a centred Gaussian process with the covariance 
\[
	e(s,t;h,A,B) = A  C_h(s,t) + B C_h(s,t), \quad A>0, B\geq 0,
\]
where $h\in (1,2)$. We already know that for $A=1, B=0$ it is a sub-fractional Brownian motion, moreover it is trivial to check that for $A=1, B=1/2$ it is a well-know fractional Brownian motion (see e.g. \cite{Samorodnitsky:1994aa}). We denote by $\xi$ a centred Gaussian process with the covariance function $e$. Such processes exhibit a long-range dependance which expresses the fact that the correlation of increments on distant intervals decays polynomially slow. We briefly summarise basic properties of $\xi$  in the following
\begin{fact} \label{fact:decay}
	Let $h \in(1,2)$. Process $\xi$ is self-similar, meaning
	\[
		\cbr{\xi(at),t\geq 0} =^d \{a^{h/2} \xi(t),t\geq 0\},\quad \text{for each }a>0. 
	\]
	Let us also denote the correlation of increments by 
	\[
		R(u,v,s,t):= \ev{}(\xi_u - \xi_v)(\xi_s - \xi_t), \quad 0\leq u<v \leq s <t,
	\]
	we have
	\begin{equation*}
		R(u,v,s+\tau,t+\tau) \sim 
		\begin{cases}
			\tau^{h-2}& \text{ if } B>0 \\
			\tau^{h-3}& \text{ if } B=0.
		\end{cases}
	\end{equation*}
\end{fact} 

\subsection{Symmetric $\alpha$-stable motion}%--------------------------------------------------------------------------------
In this section we summarise analytic properties of a symmetric $\alpha$-stable processes. This can be found in e.g.\cite{Bogdan:2009le}. Let $\alpha \in (0,2]$ and $\cbr{\eta^x_t}_t$ be the symmetric $\alpha$-stable motion in $\Rd$ starting from $x$. It is a homogenous Markov process with the transitions densities denoted by $\cbr{p_t}_{t>0}$. We also define the corresponding semigroup 
\begin{equation}
	\T{t} f(x) := \ev{} f(\eta_t^x)  = \intr p_t(x-y) f(y) \dd{y}, \quad t\geq 0,  \label{eq:semigroup-density}
\end{equation}
where $f:\R^d\mapsto \R$ is a measurable function. We have also the following property
\begin{equation}
	p_{t+s} = p_t \ast p_s,	 \label{eq:conv}
\end{equation}
where $\ast$ denotes convolution (which is nothing else but the Markov property). We will also need the following self-similarity property
\begin{equation}
	p_t(x) = t^{-d/\alpha} p_1(t^{-1/\alpha} x). \label{eq:self-similar}
\end{equation}
The Fourier transform (i.e. $\hat{f}(z) =  \intr e^{\text{i} x z} f(x) \dd{x}  $) of the semigroup $\T{}$ is
\begin{equation}
	\widehat{\T{t} f}(z) = e^{-t|z|^\alpha}\hat{f}(z) \label{eq:fourier.}.
\end{equation}
We recall that for any $f\in \SD$ we have $\hat{f} \in \SD$. It is also well-known that for any $l<\alpha$ and any $t\geq0$ we have 
\[
	\ev{} |\eta_t^x|^l <+\infty.
\]

\subsection{Infinite divisible point processes} % (fold) ----------------------------------------------------subsection--------------------------------
\label{sub:infinite_divisible_point_processes}

% subsection infinite_divisible_point_processes (end)
The point processes theory is well-developed (e.g. \cite{Kallenberg:1983pb}); in this section we excerpt facts needed in our work. We denote the space of locally finite point measures on $\Rd$ by $\mathcal{M}(\Rd)$ and endow it with the topology of vague convergence. 
 Our starting measure $M$ is an \textit{infinitely divisible point process} (IDPP). We assume that its expectation is finite and the intensity measure is $\lambda$. For any measurable function $\varphi:\R^d\mapsto \R_+$ the following L\'evy-Khinchine-like formula holds
\begin{equation}
	\lap{-\ddp{M}{\varphi} } = \exp\cbr{-\int_{\mathcal{M}(\Rd)} \rbr{1-e^{-\ddp{\varphi}{\mu}}} Q(\dd{\mu}) },  \label{eq:LevyKhinchine}
\end{equation}
where $Q$ is a measure on $\mathcal{M}(\Rd) - \cbr{0}$ which is called the canonical measure and fulfils  $\ddp{Q \pi_B^{-1}}{1-e^{-x}} <+\infty$ for any Borel bounded set $B$ and $\pi_B:\mathcal{M}(\Rd) \rightarrow \R$ is the mapping given by $\pi_B(\mu) = \mu(B)$. We associate with $Q$ a family of measures $\cbr{Q_x}_{x\in \Rd}$ called the Palm distribution which fulfils
\begin{equation}
	\int_{\mathcal{M}(\Rd)} \ddp{\mu}{h} G(\mu) Q(\dd{\mu}) = \intr h(x) \int_{\mathcal{M}(\Rd)} G(\mu) Q_x(\dd{\mu}) \dd{x},  \label{eq:palmFormula}
\end{equation}
for any $h:\Rd \rightarrow \R_+$, $G:\mathcal{M}(\Rd) \rightarrow \R_+$ which are measurable functions - see \cite[12.3]{Kallenberg:1983pb}.\\
	IDPPs may sound a bit abstract, but in fact they are not in the context of branching particle systems. We give now two examples playing a crucial role in our studies. The first is the Poisson point process (with intensity $\lambda$), for which the distribution $Q_x$ is concentrated on $\delta_x$. The second is the equilibrium measure $Eq$ defined by \eqref{eq:defEqulibrium}, its Palm measure is described in Remark \ref{rem:examples} where we will discuss these examples further. We will need shift operator $\theta_x$ defined as $\theta_x y := y-x, \: y\in \Rd$ for any  $B\subset \Rd$ and measure $\mu\in \mathcal{M}(\Rd)$ we put
\[
	\theta_x B := B-x = \cbr{y-x:y\in B},\quad \theta_x \mu(\cdot) := \mu \circ \theta^{-1}_x(\cdot).
\]
We say that a point process is simple if the multiplicity of each atom is $1$. A simple point process $M$ is fully descried by the (random) set of its atoms $[M]$, viz.
\[
	[M]:= \cbr{x\in \Rd: \delta_x \in M}.
\]

\subsection{Assumptions} % (fold)
\label{sub:assumptions}
In this section we list assumptions required further and shortly discuss them.

\begin{itemize}
	\item[(A1)] $M$ is a simple translation invariant IDPP with the finite intensity measure $\lambda$. We will denote it shortly by $\ev{M} = \lambda$ or more precisely for any $f\in \mathcal{L}^1(\Rd)$ 
	\begin{equation}
		\ev{} \ddp{M}{f} = \ddp{\lambda}{f}, \quad \ev{} \ddp{Q}{f} = \ddp{\lambda}{f}.  \label{eq:intensity}
	\end{equation}
	$M$ has a finite second moment. We will assume a condition slightly stronger then usual, viz. for any $f\in \mathcal{L}^1 $ 
	\[
		\ev{} \ddp{M}{f}^2 <+\infty.
	\]
\end{itemize}	
	For a measure $M$ fulfilling (A1) we denote by $Q$ its canonical measure. $Q$ is also translation invariant. The family of the Palm distributions of $Q$ will be denoted by $Q_x$. The translation invariance of implies that $Q_x = \theta_x Q_0$. Moreover we denote the intensity measure of $Q_x$ by $\Lambda_x := \ev{} Q_x$
	\begin{itemize}
	\item[(A2)] Let $(p_t\ast \Lambda_0)(x) = \intr p_t(x-y) \Lambda_0(\dd{y} )$. We assume that
	\[
		\sup_t t^{d/\alpha - 1}  \norm{p_t\ast \Lambda_0}{\infty}{}   < +\infty.
	\]
	\item[(A3)] There exists a non-negative number denoted by $\mathcal{H}(M)$ or $\mathcal{H}(Q)$ such that
	\[
			t^{d/\alpha -1} \ddp{p_t}{\Lambda_0} \rightarrow \mathcal{H}(Q), \quad \text{as } t\rightarrow \infty.
	\]
	
	\item[(A4)] There exist $C, \epsilon>0$ such that for any $\varphi\in \mathcal{L}^1(\Rd)$ and $x\in \Rd$ we have 
	\[
		\rbr{\int_{\mathcal{M}(\Rd)} \ddp{\varphi }{\mu}^{1+\epsilon} Q_x(\dd{\mu} )}^{1/(1+\epsilon)} \!\!\!\leq C {\int_{\mathcal{M}(\Rd)} \ddp{\varphi }{\mu} Q_x(\dd{\mu} )}.
	\]	
\end{itemize}
Let us discuss shortly the above assumptions. We skip (A2) and (A4) which are technical and not very restrictive. (A1) is quite natural as $\lambda$ is preserved by the dynamics of the branching particle system (in expectation). In other papers in the field it is common to specify only a narrow class of the starting distributions (usually a Poisson random field or the equilibrium measure $Eq$). Our approach brings a novelty by admitting a larger class of distributions (the reader should be however warned that the infinite divisibility yields a poissonian structure lurking behind the scene - see \cite[6.2]{Kallenberg:1983pb}). The detailed comparison to the known results is deferred to the next section.\\
The assumption (A3) is the most important one. When $\mathcal{H}(Q)>0$ it implies that particles appear at the beginning in ``large clans''. We postpone further discusstion of this phenomenon to Remark \ref{rem:intuitve} focusing now on giving some quantitive insights. The condition (A3) is roughly equivalent to the fact
\[
	t^{-\alpha} \Lambda_0 (B(t)) \text{ converges},
\]
where $B(t)$ is a ball of radius $r$. This equivalence is Tauberian in its spirit (though does not fall under the classical theory). For  illustrative purposes  we prove the following simple lemma, not pursuing analytical depth of the question

\begin{fact}\label{fact:measure}
	Assume that there exist $C\geq 0, \epsilon >0$ such that
	\begin{equation}
		t^{-\alpha} \Lambda_0 (B(t))  \in (\mathcal{H}(\Lambda) - C t^{-\epsilon}, \mathcal{H}(\Lambda)+C t^{-\epsilon}), \quad \forall_{t\geq1},	 \label{eq:cond2}
	\end{equation}
	then 
	\begin{equation}
		t^{d/\alpha - 1}\ddp{p_t}{\Lambda_0} \rightarrow c_\alpha \mathcal{H}(\Lambda) , \quad \text{ as }t\rightarrow +\infty,	 \label{eq:cond1}
	\end{equation}
	where $c_\alpha = \alpha \inti p_1(x) x^{\alpha -1} \dd{x}.$
\end{fact}
The proof is deferred to Appendix. 
\section{Result} % (fold)
\label{sec:result}

Let us denote convergence in law by $\rightarrow^d $ and recall the covariance functions from Section \ref{sec:gaussian}. We recall also the Gamma function $\Gamma(z)=\inti t^{z-1}e^{-t}\dd{t}$.

Below, we present a functional central limit theorem, which is the main result of the paper 
\begin{theorem}\label{thm:starting}
	Assume that $\alpha < d < 2\alpha$, and $F_T = T^{(3-d/\alpha)/2}$. Let $N$ be a branching particle system starting off from a random point process $M$ fulfilling assumptions (A1)-(A4). Let $X_T$ be its rescaled occupation time fluctuations processes given by \eqref{eq:fluctuatons_process}. We have
	\[
		X_T \rightarrow^d  X, \:\: \text{ in } \mathcal{C}([0,1], \SP) \text{ as } T\rightarrow +\infty.
	\]
	where $X$ is a measure-valued centred Gaussian process with covariance function
	\[
		Cov( \ddp{X(s)}{\varphi_1}, \ddp{X(s)}{\varphi_2} ) = \ddp{\lambda}{\varphi_1} \ddp{\lambda}{\varphi_2} \sbr{K_1 C_h(s,t) + \mathcal{H}(M)K_2 c_h(s,t)},
	\]  
	where $\mathcal{H}(M)$ is defined as in (A3), $h:=3-d/\alpha\in (1,2)$ and
	\[
		K_1 := \frac{V \Gamma(2-h)}{2^{d-1} \pi^{d/2} \alpha \Gamma(d/2) h(h-1) }, \quad K_2 := \frac{1}{h(h-1)}.
	\]
\end{theorem}

\begin{rem} \label{rem:examples}
	We now present a few examples. We keep the assumptions of Theorem \ref{thm:starting} changing only starting distributions $M$
	\begin{itemize}
		\item Let $M$ be a Poisson random field. Assumptions (A1)-(A4) are easily verifiable and  obviously $\mathcal{H}(M) = 0$ therefore limit process $X$ has  covariance function
		\[
			K_1 \ddp{\lambda}{\varphi_1} \ddp{\lambda}{\varphi_2}  C_h(s,t).
		\]
		\item Let $M$ be a ``compound Poisson random field with finite clans''. It is a random point process obtained by replacing each atom of a Poisson point process by an independent copy of another point process. More precisely, let $N$ be a simple random point process such that $\ev{}N(\Rd) = 1$ and $N(\Rd) \leq C $ for a certain constant $C>0$. Let $P$ be a Poisson point process with intensity $\lambda$. We construct $M$ by
		\[
			M:= \sum_{x \in [P]} \theta_x N^x,
		\] 
		where $N^x$ are independent copies of $N$. It is straightforward to check that $M$ is a simple, translation invariant IDPP with the canonical measure
		\[
			Q = \intr \theta_x \mathcal{N} \dd{x}, 
		\]
		where $\mathcal{N}$ is the distribution of $N$. As usual we denote the Palm distributions of $Q$ by $Q_x$. Our assumptions imply that intensity measures $\Lambda_x := \ev{} Q_x$ fulfil $\Lambda_x(\Rd)=1$. Using \eqref{eq:self-similar} we obtain
		\[
			t^{d/\alpha - 1} \ddp{p_t}{\Lambda_0} \leq t^{d/\alpha - 1} p_t(0) = t^{-1} p_1(0) \rightarrow 0.
		\]
		Therefore assumptions (A2)-(A3) are fulfilled with $\mathcal{H}(Q) = 0$. (A4) can be checked utilising condition $N(\Rd) \leq C $. Finally, using Theorem \ref{thm:starting} we conclude that limit process $X$ has covariance function
		\[
			K_1 \ddp{\lambda}{\varphi_1} \ddp{\lambda}{\varphi_2}  C_h(s,t).
		\]
		which is the same as in the case of the Poisson point process.
		\item Let $M$ be the equilibrium measure defined by \eqref{eq:defEqulibrium}. It is known that it is a simple IDPP with a finite second moment; let us denote its canonical measure by $Q$. The Palm distributions $Q_x$ are equal to the law of  $\delta_x + \xi^x_\infty$, where $\xi^x_\infty$ is the limit of
		\begin{equation}
			\xi^x_t = \int_0^t \zeta_s^{s,X_s^x} \nu(\dd{s}), \label{eq:clan}
		\end{equation}
		where $\nu$ is a ``Poisson clock'' with intensity $V$, $\cbr{X_s^x}_{t\geq 0}$ a symmetric $\alpha$-stable process starting at $x$ (we think of it as moving backward in time), $\cbr{\zeta_{t}^{s,y}}_{t\geq 0}$ are independent branching particle systems starting with one particle at $y$ at time $-s$. $\xi^x_t$ is a clan of particles having a common ancestor, for details and formalisation see \cite{Gorostiza:1991aa,Zahle:2002aa}. We prove in Section \ref{sec:equlibrium} that the intensity measure of $\xi^0_\infty$ is ($\delta_0$ if irrelevant in limit)
		\[
			\Lambda_0(\dd{y} ) = \frac{V c_{\alpha,d}}{2} |y|^{\alpha -d} \dd{y}. 
		\]
	Using \eqref{eq:self-similar} we can make explicit calculations (we denote $C_1 := \frac{V c_{\alpha,d}}{2} $)
		\[
			t^{d/\alpha-1}(p_t \ast \Lambda_0)(x) = C_1 t^{d/\alpha-1} \intr p_t(y)  |x-y|^{\alpha -d}  \dd{y} = C_1 t^{-1}\intr p_1( t^{-1/\alpha }y) |x-y|^{\alpha-d} \dd{y} =(*).
		\]
		  Now changing integration variable $y \rightarrow t^{1/\alpha} y$ we get
		\[
			(*) = C_1 t^{d/\alpha-1}\intr p_1( y) |x-t^{1/\alpha }y|^{\alpha-d}  \dd{y} =  C_1 \intr p_1( y) |t^{-1/\alpha }x-y|^{\alpha-d} \dd{y}.
		\]
		Using the last expression we easily verify (A2) and pass to the limit
		\[
			t^{d/\alpha-1}\ddp{p_t}{\Lambda_0} = t^{d/\alpha-1}(p_t \ast \Lambda_0)(0) \rightarrow \frac{Vc_{\alpha,d}}{2} \intr  \frac{p_1( y)}{|y|^{d-\alpha}} \dd{y} =: \mathcal{H}(Q). %= \frac{1}{2}c_{\alpha,d}  2 \frac{\pi^{d/2}}{\Gamma(d/2)} \alpha^{-1} \Gamma(d/\alpha -1) = \Gamma\rbr{\frac{d-\alpha}{2}} (2^\alpha \pi^{d/2} \Gamma(d/2))^{-1}   \frac{\pi^{d/2}}{\Gamma(d/2)} \alpha^{-1} \Gamma(d/\alpha -1)
		\]
		We defer calculation of $\mathcal{H}(Q)$ to \eqref{eq:tmpIntegral}. Further, using Theorem \ref{thm:starting} we obtain that the limit process has covariance function
		\[
			K_1 \ddp{\lambda}{\varphi_1} \ddp{\lambda}{\varphi_2}  \sbr{C_h(s,t) + 1/2 c_h(s,t)}.
		\]
		We have already noticed that $C_h(s,t) + 1/2 c_h(s,t)$ is the covariance function of the sub-fractional Brownian motion.\\ A careful reader notices that we did not check (A4). Unfortunately, because of lack of manageable moments of order $1+\epsilon$ it is very hard. Instead, using techniques of \cite{Milos:2007ab} one can check directly that \eqref{eq:z32Z3} converges to $0$. This is the only place where (A4) is used.
	\end{itemize}	
\end{rem}

\begin{rem} \label{rem:intuitve}
	In this remark we would like to illustrate the theorem in an intuitive manner thus we do not pursue full mathematical rigour.\\
	(a)  We were able to indicate parameter $\mathcal{H}(M)$ of the starting distribution, which is crucial for the behaviour of limiting process. Measures $Q$ and $Q_x$ describe particles which together enter the system at time $t=0$. We will call such group of particles a clan. For rigorous treatment of clans in case of $Eq$ see \cite{Zahle:2002aa}. \\
	(b) Assumption (A3) gives a condition required to clans to have ``influence to the limit''. In order to explain it we refer to Fact \ref{fact:measure}. Under assumption $\mathcal{H}(\Lambda)$ equation \eqref{eq:cond2} states that the (expected) number of particles in a clan has to be of the specified density. In particular this implies that the clan has to be infinite. \\
	(c) The limit process $X$ could be written in the form $X = \xi \lambda$, where $\xi$ is a real-valued centred Gaussian process with covariance function $K_1 C_h(s,t) + \mathcal{H}(M)K_2 c_h(s,t)$. It has a simple spatial structure and a complicated temporal one. Let us recall that $\xi$ is a long-range dependance process with the covariance decay rate described by Fact \ref{fact:decay}. The long range dependance stems from the recurrence properties of the movement. This could be considered at particles' level and at clans' level. The former can be explained as follows, a particle which occupies a given set at time $t_1$ is likely to ``still be there (or come back)'' at some later time $t_2$. Of course the likelihood decreases with $t_2 - t_1$ (our results indicate that the decay order is $-d/\alpha$). The clans' level dependance is of slightly different mechanism. Again we consider a particle which visits a given set at time $t_1$. It is accompanied by a large number of relatives from its clan (clans are infinite!). So even when it fails to be in the set at time $t_2$ there is still a chance that some of its kinsman is there.  This is reasoning is delicate, let us notice that the clans members evolve independently and the only source of dependance is through the fact that they arrived at $t=0$ in a cluster. Since the clan is infinite the family bonds are more strongly correlated than particles itself (i.e. when $\mathcal{H}(M)>0$  the decay order is $-d/\alpha+1$).\\
	(d) The above explanation could be put in a slightly more formal setting (the reader might find it useful to read the proofs first and come back to this remark later). By the polarisation formula we have $$Cov(\ddp{X_T(s)}{\varphi}, \ddp{X_T(t)}{\varphi}) = \sbr{L(s,0,T)+L(t,0,T)-2L(s,t,T)}/2,$$
	where $L$ is given by \eqref{eq:l1}. One is further referred to \eqref{eq:l2}. The term $J_3$ capture the dependance at the particle level and results in $C_h(s,t)$. The terms $J_1,J_2$ describe the dependance on the clan level. To see we consider the ``clan-less configuration'', viz. the Poisson point process. One checks that $Q = \intr \delta_x \dd{x}$ and hence $J_1+J_2=0$. $J_2$ turns out to be negligible in limit. Finally we draw the reader attention to \eqref{eq:l3} which could be proved to converge to $\mathcal{H}(Q)c_h(s,t)$.\\
	(e) Another way of explaining the additional term is to notice that $\ddp{n_T}{\mu}$ in $J_1$ is the same as 
	$$\intr \int_{Tt_1}^{Tt_2} \T{t}\varphi_T(x) \dd{t}   \mu(\dd{x}) = \int_{Tt_1}^{Tt_2}\ddp{\varphi_T}{ p_t \ast \mu}\dd{t} = \int_{Tt_1}^{Tt_2}\ev{} \ddp{\varphi_T}{ N_t^\mu}\dd{t},$$
		where $N^\mu$ denotes the branching system starting from $N_0 = \mu$. Therefore $\ddp{n_T}{\mu}$ could be considered as the occupation time of the mean density of the $N^\mu$ (that is the mean density of a clan). In this interpretation, $J_1+J_2$ is a variance, and consequently the term $\mathcal{H}(M)K_2 c_h(s,t)$ stems from the variability of the initial clan configurations. \\
		(f) We notice also that the limit process does not have stationary increments unless $\mathcal{H}(M) = \frac{2K_2}{K_1}$. We already know that this happens in case of $Eq$ but now we see that it can happen more generally. We suspect that this may be related to the speed of convergence to $Eq$.
\end{rem}
\begin{rem}
	The limit process $X$ could be written in the form $X = \xi \lambda$, where $\xi$ is a real-valued centred Gaussian process with covariance function $K_1 C_h(s,t) + \mathcal{H}(M)K_2 c_h(s,t)$. We observe a simple spatial structure and a complicated temporal one. Let us recall that $\xi$ is a long-range dependance process with the covariance decay rate described by Fact \ref{fact:decay}. The form of the limit is intuitively explained with the notion of clans recurrence. In \cite{Stoeckl:1994aa} properties of clans are described in the case of $Eq$ measure. A clan here can be easily imagined once we consider an ``infinitely old'' branching particle system  $\cbr{N_t}_{t \in \R}$. At time $0$, the clans are maximal sets of particles having a common ancestor in past. For the system of Brownian particles in $d=3$ (i.e. the only intermediate dimension for $\alpha=2$) any given ball is visited infinitely often by members of each clan. This results in smoothing of the spatial structure but introduces long-range dependance to the temporal one. The clans have not been studied in the case of $\alpha$-stable particle systems and measures different then $Eq$ but we believe that the picture is still valid.
\end{rem}

\begin{rem} \label{rem:bojecki}
	 This research was inspired by developments of \cite{Bojdecki:2010bl}. They consider the same problem for systems of particles moving in $\R$ and a class of starting measures defined by
	\[
		M = \sum_{j \in \mathbb{Z}} \sum_{n=1}^{\theta_j} \delta_{\rho^j_{\theta_j,n}},
	\] 
	where $\theta_j$ are independent copies of a non-negative integer-valued random variable $\theta$ fulfilling $\ev{}\theta^3 <+\infty$ and for each $j\in \mathbb{Z}$ and $k=1,2,\ldots$ we have a random vector $\rho_k^j = \rbr{\rho_{k,1}^j, \rho_{k,2}^j, \ldots, \rho_{k,k}^j }$ with values in $[j,j+1)^k$ such that $\rbr{\theta_j, (\rho_k^j)_{k=1,2, \ldots}}_{j\in \mathbb{Z}}$ are independent. They proved that for the system of particles starting from $M$ and moving according to a symmetric $\alpha$-stable motion \emph{without branching} the limit process $X$ is a centred Gaussian process with covariance function (Theorem 2.4)
	\[
		C \ddp{\lambda}{\varphi_1}\ddp{\lambda}{\varphi_2} \sbr{({\ev{\theta}}) C_h(s,t) + 1/2 \rbr{\var \theta} c_h(s,t)}.
	\]
	This result is very similar to ours with the role of $\mathcal{H}(Q)$ played by $\var \theta$. From this point of view the interpretation given in Remark \ref{rem:intuitve} (d) seems appealing. In \cite{Bojdecki:2010bl} the system with branching is also studied but in this case the limit depends only on $\ev{\theta}$. As the authors explain in Remark 2.5 ``since the trees tend to extinction by criticality, the effect of the fluctuations of $\theta$ become more and more negligible  ...''. Our results could be seen as complementary to this fact and assumption (A3) is a condition which prevents the fluctuations to disappear.
\end{rem}
\begin{rem}
	A natural question appears what happens when the limit in (A3) is infinite i.e. the clans are infinite and ``denser'' that \eqref{eq:cond2}. We conjecture that in such a case one needs norming larger than $F_T = T^{(3-d/\alpha)/2}$ and the limit process $X$ is a centred Gaussian process with covariance function
	\[
		C \ddp{\lambda}{\varphi_1}\ddp{\lambda}{\varphi_2} c_h(s,t),
	\]
	where $C>0$. That would indicate that the condition of assumption (A3) defines a point of a  ``phase transition'' from the case of small clans and processes with covariance function $C_h$ to the case of big clans and processes with covariance function $c_h$.
\end{rem}
\begin{rem}
	\cite{Milos:2007ab} covers a wider class of branching law, namely critical branching laws with finite variance. We conjecture that our theorem  holds in this setting (with possible change of constants) too. \\
	% (A1) contains the assumption of the finite variance of the starting distribution. It was used only in the prove tightness. We expect that it may be weaken. \\
	We also suspect that analogous results hold for systems with infinite variance branching law (like the ones in \cite{Bojdecki:2007aa, Milos:aa}).
\end{rem}
% section result (end)

\section{Proofs} % (fold)
\label{sec:proofs}
We follow the path, which proved to be successful in solving similar problems. It stems from \cite{Bojdecki:2006ab}. The description below is based on \cite[Section 3.1, Section 3.2]{Mios:2009oq}.
\subsection{Scheme} \label{sec:scheme}
To make the proof more clear we present a general scheme first and defer details to separated sections. Although the processes $X_T$ are signed-measure-valued it is convenient to regard them as processes with values in $\SP$. In this space one may employ a space-time method introduced by \cite{Bojdecki:1986aa} which together with Mitoma’s theorem constitute a powerful technique in proving weak, functional convergence. For a continuous, $\SP$-valued process $X= \cbr{X_t }_{t\geq0}$ we define an $\mathcal{S}'(\mathbb{R}^{d+1})$-valued random variable $\tilde{X}$ by
\begin{equation}
 \ddp{\tilde{X}}{\Phi} := \intc{1}\ddp{X_t}{\Phi(\cdot,t)}\dd{t}, \label{eq:space-time-method}
\end{equation}
we will refer to $\tilde{X}$ as the space-time variable of $X$.
\paragraph*{Convergence} From now on we will denote by $\tilde{X}_T, \tilde{X}$ the space-time variables corresponding to $X_T, X$. To prove convergence of $\tilde{X}_T$ we will use its Laplace functional
\begin{equation*}
 L_T(\Phi) = \ev{e^{-\ddp{\tilde{X}_T}{\Phi}}},\quad \Phi\in \mathcal{S}(\mathbb{R}^{d+1})_+.
\end{equation*}
For the limit process $X$ denote
\begin{equation*}
 L(\Phi) = \ev{e^{-\ddp{\tilde{X}}{\Phi}}},\quad \Phi\in \mathcal{S}(\mathbb{R}^{d+1})_+.
\end{equation*}
Once we establish the convergence 
\begin{equation}
 L_T(\Phi) \rightarrow L(\Phi),\quad \text{ as } T\rightarrow +\infty, \forall_{\Phi\in \mathcal{S}(\mathbb{R}^{d+1})_+}. \label{lim:laplace}
\end{equation} 
we will obtain weak convergence $\tilde{X}_T \rightarrow^d  \tilde{X}$. Two technical remarks should be made here. We consider only non-negative $\Phi$. The procedure how to extend the convergence to any $\Phi$ is explained in \cite[Section 3.2]{Bojdecki:2006ab}. Another issue is the fact that $\ddp{\tilde{X}_T}{\Phi}$ is not non-negative (which is a usual condition to use the Laplace transform). The usage of the Laplace transform in this paper is justified by the Gaussian form of the limit. For more detailed explanation one can check \cite[Section 3.2]{Bojdecki:2006ab}. Detailed calculations for this part of the scheme will be conducted in Section \ref{sub:the_laplace_transform}.
\paragraph*{Tightness} By \cite{Bojdecki:1986aa} we know that $\tilde{X}_T \rightarrow^d  \tilde{X}$ together with tightness of $\cbr{X_T}_{T}$ implies convergence $X_T \rightarrow^d X $ in $\mathcal{C}([0,1], \SP)$. The tightness can be proved utilising the Mitoma theorem \cite{Mitoma:1983aa}. It states that tightness of $\cbr{X_T}_T$ with trajectories in $C([0, 1], \SP)$ is equivalent to tightness of $\ddp{X_T}{\varphi}$, in $\mathcal{C}([0,\tau],\mathbb{R})$ for every $\varphi \in \mathcal{S}(\Rd)$. Using the simple trick (see \cite{Bojdecki:2006ab}[Lemma on p.9]) we may also assume that $\varphi\geq 0$. Let us recall a classical criterion \cite[Theorem 12.3]{Billingsley:1968aa}, i.e. a process $\cbr{\ddp{X_T(t)}{\varphi}}_{t\geq 0}$ is tight if for any $t_1,t_2\geq0$ and there exist  constants $C>0$, $\gamma>1$
\begin{equation}
 \ev{(\ddp{X_T(t_1)}{\varphi}- \ddp{X_T(t_2)}{\varphi})^2} \leq C (t_1-t_2)^\gamma. \label{ineq:tightness}
\end{equation}
% Following the scheme in \cite{Bojdecki:2006aa} we define a sequence $(\psi_n)_n$ in $\mathcal{S}(\mathbb{R})$, and $\chi_n(u) = \int_u^1 \psi_n(s) \dd{s}$ in such a way that 
% \begin{equation}
%  \psi_n\rightarrow \delta_t - \delta_s,\quad 0\leq \chi_n \leq \mathbf{1}_{[s,t]}.\label{tmp:chi-n}
% \end{equation}
% Denote $\Phi_n=\varphi\otimes\psi_n$. We have
% \[
%  \lim_{n\rightarrow +\infty} \ddp{X_T}{\Phi_n} = \ddp{X_T(t)}{\varphi} - \ddp{X_T(s)}{\varphi}
% \]
%  thus by the Fatou lemma and the definition of $\psi_n$ we will obtain (\ref{ineq:tightness}) if we prove that
% \[
%  \ev{\ddp{\tilde{X}_T}{\Phi_n}^2} \leq C (t-s)^\gamma,
% \]
% where $C$ is a constant independent of $n$ and $T$. From now on we fix an arbitrary $n$ and denote $\Phi:=\Phi_n$ and $\chi := \chi_n$. 
% By properties of the Laplace transform we have
% \[
%  \ev{\ddp{\tilde{X}_T}{\Phi}^2} = \left.\frac{d^2}{d\theta^2}\right|_{\theta = 0}\!\! \ev{\eexp{-\theta \ddp{\tilde{X}_T}{\Phi}}}
% \]
% Hence the proof of tightness will be completed if we show 
% \begin{equation}
%  \left. \frac{d^2}{d\theta^2}\right|_{\theta = 0}\!\! \ev{\eexp{-\theta \ddp{\tilde{X}_T}{\Phi}}} \leq C(t-s)^\gamma. \label{eq: laplace-tighness-end}
% \end{equation}
Calculations are deferred to Sections \ref{sec:tightness}.

\subsection{One-particle equation}
In this section we will introduce an equation which is crucial for further analysis. The law of the occupation time for the system starting off from a single particle at $x$ is described by 
\begin{equation}
v_{\Psi}\left(x,r,t\right):=1-\mathbb{E}\exp\left\{ -\int_{0}^{t}\left\langle N_{s}^{x},\Psi\left(\cdot,r+s\right)\right\rangle \dd{s}\right\}, \Psi\geq0 \label{def:v},
\end{equation}
where $N_s^x$ denotes the empirical measure of the particle system with the initial condition $N_0^x = \delta_x$. More precisely $N^x$ is a system starting from one particle placed at $x$ and dynamics described in the Introduction. The following lemma gives the announced equation
\begin{lem} \label{lem:equation} %-------------------------------Lemma-begin----------------------------------------%
 Assume that $\Psi \in \mathcal{S}(\mathbb{R}^{d+1})_+$ then
\begin{equation}
 0\leq v_\Psi\leq 1, \label{eq:vin[0,1]}
\end{equation} 
and $v_\Psi$ satisfies the equation
\begin{equation}
v_{\Psi}\left(x,r,t\right)=\int_{0}^{t}\mathcal{T}_{t-s}\left[\Psi\left(\cdot,r+t-s\right)\left(1-v_{\Psi}\left(\cdot,r+t-s,s\right)\right)-\frac{V}{2}  v_{\Psi}^2\left(\cdot,r+t-s,s\right)\right]\left(x\right)\dd{s}.\label{eq:v-integral}
\end{equation}
\end{lem} %-------------------------------Lemma-end----------------------------------------%
The proof follows by use of the Feynman-Kac formula and can be find in \cite[Section 3.3]{Bojdecki:2006ab}. $v_\Psi$ is quite cumbersome to deal with hence we approximate it with $n_\Psi$ defined by
\begin{equation}
 n_{\Psi}(x,r,t) := \intc{t} \T{t-s} \Psi(\cdot,r+t-s) \dd{s}, \quad \Psi \in \mathcal{S}(\mathbb{R}^{d+1} ), x\in \Rd, r,t\geq0.\label{sol:tv}
\end{equation} 
Intuitively, $n_\Psi$ was obtained by dropping quadratic terms in \eqref{eq:v-integral} as they do not play a role when $\Psi$ is small. 
\paragraph{Notation} From now on we fix non-negative $\Phi\in \mathcal{S}(\R^{d+1})_+$ and prove convergences announced in the scheme in Section \ref{sec:scheme}. To make the proof shorter we will consider $\Phi$ of a special form.  The proof for general $\Phi$ goes the same lines but is more notationally cumbersome
\begin{equation}
 \Phi(x,s) := \varphi(x)\psi(s),  \quad  \varphi\in \SD_+,\:\psi \in \mathcal{S}(\mathbb{R})_+. \label{def:simplification}
\end{equation}
We also denote
\begin{equation}
 \varphi_{T}\left(x\right):=\frac{1}{F_{T}}\varphi\left(x\right), \:\chi(s) := \int_s^1\psi(u) \dd{u},\: \chi_{T}:=\chi\left(\frac{t}{T}\right). \label{def:notation-T}
\end{equation}
We write
\begin{equation*}
 \Psi(x,s) := \varphi(x)\chi(s),
\end{equation*}
\begin{equation}
  \Psi_{T}\left(x,s\right):=\frac{1}{F_{T}}\Psi\left(x,\frac{s}{T}\right) = \varphi_T(x)\chi_T(s). \label{def:simpl}
\end{equation}
Note that $\Psi$ and $\Psi_T$ are positive functions.
In the sequel, we also write
\begin{equation}
 v_T(x,r,t):=v_{\Psi_T}(x,r,t) \:\text{ and }\: v_T(x) := v_T(x,0,T), \label{eq:simpl1}
\end{equation}
and we define in the same way $n_T(x,r,t), n_T(x)$. Finally, by \eqref{eq:v-integral}, \eqref{sol:tv} for $\Psi\geq 0$ and by the fact that the potential operator of $\alpha$-stable motion exists for $d>\alpha$, it is straightforward to check 
\begin{equation}
  0\leq v_T(x) \leq n_T(x) \leq \frac{C_{\Psi}}{F_T},\quad \forall_{x\in \Rd}, \label{ineq:tv>v}
\end{equation}
and some $C_\Psi\geq 0$. We will also use the following simple estimate 
\begin{equation}
	n_T(x) - v_T(x) \cleq \frac{1}{F_T} n_T(x) +  \intc{T} \T{T-s} n_T^2(x,T-s,s) \dd{s}.\label{eq:n-v}
\end{equation}
which again follows from \eqref{eq:v-integral}, \eqref{sol:tv} and \eqref{ineq:tv>v}.
Let us fix $\Psi \in \SD_+$ and denote $v(\theta) = v_{\theta \Psi}$. Obviously $v(0)=0$. In the sequel we will need derivatives of $v$ with respect to $\theta$. Using \eqref{eq:v-integral} it is easy to calculate that (we omit arguments and integration variables) 
\begin{equation}
 v'(\theta) = \int_{0}^{t}\mathcal{T}_{t-s}\left[\Psi \left(1-v(\theta)\right)-\theta \Psi v'(\theta) -   V v(\theta)v'(\theta)  \right] \dd{s} .\label{eq:trata1}
\end{equation} 
When $\theta=0$ we have
\begin{equation}
 v'(0)(x,r,t) = \int_{0}^{t}\mathcal{T}_{t-s}\Psi(x,r+t-s) \dd{s} = n_\Psi(x,r,t). \label{eq:vprim}
\end{equation} 
We will also need the second derivative 
\[
	v''(\theta) = \int_{0}^{t}\mathcal{T}_{t-s}\left[ -2 \Psi v'(\theta) - \theta \Psi v''(\theta) -  V v(\theta)v''(\theta) -  V (v'(\theta))^2  \right] \dd{s}.
\]
When $\theta=0$ we have
\begin{equation}
	v''(0)(x,r,t) = - \int_{0}^{t}\mathcal{T}_{t-s}\left[2\Psi(\cdot,r+t-s) n_\Psi(\cdot, r+t-s,s) + V n_\Psi(\cdot, r+t-s,s)^2  \right] \dd{s}. \label{eq:vpp}
\end{equation}
By \eqref{def:v} and properties of the Laplace transform we have
\begin{equation*}
	\ev{} \rbr{\int_{0}^{t}\left\langle N_{s}^{x},\Psi\left(\cdot,s\right)\right\rangle \dd{s}}^2 = -v''(0)(x,0,t).
\end{equation*}
Formally the equation is true for $\Psi \in \mathcal{S}(\mathbb{R}^{d+1})_+$. Exploiting the Lebesgue dominated theorem one can easily argue that it is also true for $\Psi$ of the form $\Psi(x,s) = \varphi(x) 1_{[t_1, t_2]}$ yielding 
\begin{equation}
	\ev{} \rbr{\int_{t_1}^{t_2}\left\langle N_{s}^{x},\varphi \right\rangle \dd{s}}^2 = -v''(0)(x,0,t), 	 \label{eq:secondmome}
\end{equation}
with $v''$ defined according to \eqref{eq:vpp} with this special $\Psi$.

\subsection{Miscellaneous}
In this section we gather some simple facts used in the proof below. All functions in this section as assumed to be measurable
\begin{lemma}\label{lem:a2reformulations}
	Let us assume (A2). There exists a constant $C>0$ such that for any $\varphi:\R^d \mapsto \R_+$ we have
	\[
		\sup_t t^{d/\alpha-1} \int_{\mathcal{M}(\Rd)} \ddp{\T{t} \varphi }{\mu} Q_0(\dd{\mu} ) < c \norm{\varphi}{1}{}.
	\]
\end{lemma}
\begin{proof}
	Using \eqref{eq:semigroup-density} we have
	\begin{multline*}
			\int_{\mathcal{M}(\Rd)} \ddp{\T{t} \varphi }{\mu} Q_0(\dd{\mu} ) = \intr \T{t} \varphi(x) \Lambda_0( \dd{x}  ) = \intr \intr p_t(x-y) \varphi(y) \dd{y}  \Lambda_0( \dd{x}  ) \\
			= \intr  \varphi(y) (p_t\ast \Lambda_0)( \dd{y}  ) \leq \norm{p_t\ast\Lambda_0}{\infty}{} \norm{\varphi}{1}{}.
	\end{multline*}
	Now the proof follows directly by (A2).
\end{proof}

\begin{lemma}\label{lem:convergence}
	Let us assume (A2) and $n_T$ be given by \eqref{eq:simpl1} then
	\[
		\ddp{n_T}{\mu} \rightarrow  0 \text{ in } \mathcal{L}^1(Q_x).
	\]
	Moreover the convergence is uniform in $x$.
\end{lemma}
\begin{proof} Using similar calculations as in the previous lemma we have
	\begin{multline*}
		\ddp{n_T}{\mu}=\int_{\mathcal{M}(\Rd)}\ddp{n_T}{\mu} Q_x(\mu) = \int_{\Rd} n_T(y) \Lambda_x(\dd{y} ) =  \int_{\Rd}\intc{T} \T{s}\varphi_T(y) \dd{s}  \Lambda_x(\dd{y} ) \\=  \int_{\Rd}\intc{T} \T{s}\varphi_T(x+y) \dd{s}  \Lambda_0(\dd{y} )  
		= F_T^{-1}  \intc{T}  \int_{\Rd}\varphi(x+y)  (p_s\ast\Lambda_0)(\dd{y} ) \dd{s} \cleq F_T^{-1} \intc{T}  s^{1-\frac{d}{\alpha}} \dd{s} \\
		\ceq F_T^{-1} T^{2-\frac{d}{\alpha}} = T^{(1-\frac{d}{\alpha})/2} \rightarrow 0. 
	\end{multline*}
\end{proof}

\begin{lemma}\label{lem:derivation}
	Let $f_1,f_2:\Rd\mapsto \R_+$. Then 
	\[
		\intr  \int_{\mathcal{M}(\Rd)}  \T{s_1} f_1(x)   \ddp{ \T{s_2} f_2   }{\mu}   Q_x(\dd{\mu}) \dd{x} = \int_{\mathcal{M}(\Rd)} \ddp{\T{s_1+s_2}F^*}{\mu}   Q_0(\dd{\mu}),
	\]
	where $F^*(z) :=\intr f_1(-z+z_2)f_2(z_2) \dd{z_2}$.
\end{lemma}
\begin{proof}
	We recall that $Q_x = \theta_x Q_0$, hence
	\begin{equation*}
		 \intr  \int_{\mathcal{M}(\Rd)}  \T{s_1} f_1(x)   \ddp{ \T{s_2} f_2   }{\mu}   Q_x(\dd{\mu}) \dd{x} =  \intr  \int_{\mathcal{M}(\Rd)}  \T{s_1} f_1(x)   \ddp{ \T{s_2} f_2   }{\theta_x \mu}   Q_0(\dd{\mu}) \dd{x} =(*). 
	\end{equation*}
	Using \eqref{eq:semigroup-density}, the fact that $Q_0$ is a simple point random measure and the Fubini theorem again we obtain 
	\[
	(*)=  \int_{\mathcal{M}(\Rd)} \sum_{y \in [\mu]} \intr \intr \intr p_{s_1}(x-z_1)  f_1(z_1) p_{s_2}(x+y-z_2)  f_2(z_2) \dd{x} \dd{z_1} \dd{z_2}     Q_0(\dd{\mu}) =(*).
	\]
	In the next step we substitute $x\rightarrow x+z_1$, use the semigroup property \eqref{eq:conv} and symmetry of $p_s$ to  obtain
	\begin{multline*}
		(*) = \int_{\mathcal{M}(\Rd)} \sum_{y \in [\mu]} \intr \intr \intr p_{s_1}(x)  f_1(z_1) p_{s_2}(x+y+z_1-z_2)  f_2(z_2) \dd{x} \dd{z_1} \dd{z_2}     Q_0(\dd{\mu}) \\
		=\int_{\mathcal{M}(\Rd)} \sum_{y \in [\mu]} \intr \intr f_1(z_1) p_{s_1+s_2}(x+y+z_1-z_2)  f_2(z_2) \dd{x} \dd{z_1} \dd{z_2} Q_0(\dd{\mu}) =(*).
	\end{multline*}
	Further, substitute $z_1\rightarrow z_1 +z_2$ and  we write 
	\begin{multline*}
		(*)= \int_{\mathcal{M}(\Rd)} \sum_{y \in [\mu]} \intr \intr   f_1(z_1+z_2) p_{ (s_1+s_2)}(y+z_1)  f_2(z_2)  \dd{z_1} \dd{z_2}     Q_0(\dd{\mu})  \\
		=\int_{\mathcal{M}(\Rd)} \sum_{y \in [\mu]} \intr   p_{ (s_1+s_2)}(y+z_1)  F^*(-z_1)  \dd{z_1}     Q_0(\dd{\mu}) = \int_{\mathcal{M}(\Rd)} \ddp{\T{s_1+s_2}F^*}{\mu}   Q_0(\dd{\mu}). 
	\end{multline*}
\end{proof}

We present also a simple fact concerning moments of a IDPP random measures. 
\begin{lemma}\label{lem:ippdSecondMoment}
	Let $M$ be a IDPP random measure with canonical measure $Q$. Let $f:\Rd\mapsto \R_+$ be such that $\ev{} \ddp{M}{f}^2 < +\infty$ then
	\[
		\ev{} \ddp{M}{f}  = \int_{\mathcal{M}(\Rd)} \ddp{f}{\mu} Q(\mu).
	\]
	\[
		\ev{} \ddp{M}{f}^2  = \rbr{\int_{\mathcal{M}(\Rd)} \ddp{f}{\mu} Q(\mu) }^2 + \rbr{\int_{\mathcal{M}(\Rd)} \ddp{f}{\mu}^2 Q(\mu) }.
	\]
\end{lemma}
We omit the proof which follows by standard application of the properties of the Laplace transform.

% subsection misc (end)
% 
% subsection scheme (end)
\section{Calculations} \label{sec:calculations}
In this section we execute calculations announced in the proof scheme described in Section \ref{sec:scheme}
\subsection{The Laplace transform} % (fold)
\label{sub:the_laplace_transform}
In this section we calculate the Laplace transform of the space-time variable corresponding to $X_T$. Let us recall definitions \eqref{eq:fluctuatons_process} and \eqref{eq:space-time-method}. Performing simple calculations we obtain
\[
	L(\Phi) = \lap{-\ddp{\tilde{X}_T}{\Phi}} =  \exp\cbr{\intc{T} \ddp{\Psi_T(\cdot, s)}{\lambda} \dd{s}} \lap{ -\intc{T} \ddp{N_s}{\Psi_T(\cdot, s)} \dd{s}  }.
\]
Next we condition with respect to the starting measure $M$ and apply \eqref{def:v}
\begin{multline*}
 {\ev{\rbr{\left.\exp{\cbr{-\intc{T}\ddp{N_{s}}{\Psi_T(\cdot,s)}\dd{s}}}\right|M}}} = \prod_{x \in [M]} \ev{\exp{\cbr{-\intc{T}\ddp{N^x_{s}}{\Psi_T(\cdot,s)}\dd{s}}}} = \prod_{x \in [M]} \rbr{1-v_{T}(x)}.\quad\quad\quad
\end{multline*}
Now using identity $\prod_{x \in [M]} \rbr{1-v_{T}(x)} = \exp\cbr{\ddp{M}{\ln(1-v_T)}}$ and \eqref{eq:LevyKhinchine} we obtain
\[
	L(\Phi) = \exp\cbr{\intr \intc{T} \Psi_T(x,s) \dd{s} \dd{x}}\exp\cbr{-\int_{\mathcal{M}(\Rd)} \rbr{1-e^{\ddp{\ln(1-v_T)}{\mu}}} Q(\dd{\mu}) }.
\]
Let us denote the exponent of the expression above by $I(T)$, viz. $\lap{-\ddp{\tilde{X}_T}{\Phi}}= \exp\cbr{I(T)}$.
We introduce the following decomposition
\begin{multline*}
	I(T) = \underbrace{\intr \intc{T} \Psi_T(x,s) \dd{s} \dd{x}- \intr v_T(x) \dd{x}}_{I_1(T)}  + \underbrace{\intr  \ln(1-v_T(x))+v_T(x) \dd{x}}_{I_2(T)} \\
 	- \underbrace{\int_{\mathcal{M}(\Rd)} \rbr{1+\ddp{\ln(1-v_T)}{\mu}- e^{\ddp{\ln(1-v_T) }{\mu}} } Q(\dd{\mu})}_{I_3(T)},
\end{multline*}
where between the lines we used the identity $\int_{\mathcal{M}(Rd)} \ddp{f}{\mu} = \intr f(x) \dd{x} $ ($\lambda$ is the intensity measure of $Q$). The proof will be completed once we obtain
\[
	I_1(T) \rightarrow \frac{K_1}{2} \intc{1} \intc{1} \intr \intr \Phi(x,t) \Phi(x,t) \dd{x} \dd{y} C_h(s,t) \dd{s} \dd{t}   ,\: \text{ as } T\rightarrow \infty. 
\] 
\begin{equation*}
	I_2(T) \rightarrow 0,\: \text{ as } T\rightarrow \infty. 
\end{equation*}
\begin{equation}
	I_3(T) \rightarrow -  \mathcal{H}(Q) \frac{K_2}{2} \intc{1} \intc{1} \intr \intr \Phi(x,t) \Phi(x,t) \dd{x} \dd{y} c_h(s,t) \dd{s} \dd{t},\: \text{ as } T\rightarrow \infty.  \label{eq:I4conv}
\end{equation}
where $K_1$ and $K_2$ are the same as in Theorem \ref{thm:starting}.   For now using the above limits we deduce
\[
	I(T) \rightarrow \frac{1}{2}\intc{1} \intc{1} \intr \intr \Phi(x,t) \Phi(x,t) \dd{x} \dd{y}\sbr{ K_1 C_h(s,t) + \mathcal{H}(Q) K_2 c_h(s,t)} \dd{s} \dd{t}.
\] 
It is easy to check that this concludes the proof of \eqref{lim:laplace}. Therefore according to scheme from Section \ref{sec:scheme} we are left of the proof of tightness. This is deferred to Section \ref{sec:tightness} and now we proceed to the proof of the converges announced above.
% subsection the_laplace_transform (end)

% section proofs (end)
\paragraph{Convergence of $I_1$} % (fold)
\label{par:paragraph_name}
This convergence can be recovered easily by inspecting \cite[(3.23)-(3.30)]{Bojdecki:2006ab}.
% paragraph paragraph_name (end)
\paragraph{Convergence of $I_2$} 
Inequality \eqref{ineq:tv>v} entitles us to use of the inequality $|\ln(1-x)+x| \leq x^2, x\in(0,1/2)$. Whence by \eqref{ineq:tv>v} and \eqref{sol:tv} we have
\[
	|I_2(T)| \cleq \intr  v_T^2(x) \dd{x} \leq \intr  n_T^2(x) \dd{x} \cleq \intr \rbr{\intc{T} \T{s} \varphi_T(x) \dd{s}}^2  \dd{x}.
\]
Using \eqref{eq:fourier.}, rearranging, using the Fubini theorem and recalling \eqref{def:notation-T} we arrive at
\[
	|I_2(T)| \cleq \intr |\widehat{\varphi}_T(z)|^2 \intc{T} \intc{T} e^{-s_1|z|^\alpha} e^{-s_2|z|^\alpha} \dd{s_1} \dd{s_2} \dd{z} \leq F_T^{-2}   \intr  \frac{|\widehat{\varphi}(z)|^2}{|z|^{2\alpha}} \rbr{1-e^{-T|z|^\alpha}}^2\dd{z}. 
\]
We now introduce an inequality which will be used a few times. Namely for any $\delta \in (0,1]$ we have
\begin{equation}
	1-e^{-x} \leq x^\delta,\: \forall_{x>0}. \label{eq:ineq-exp}
\end{equation}
Now using the inequality with $\delta = 1/2$ and recalling that $F_T^2 = T^{3-d/\alpha}$ we obtain
\[
	|I_2(T)|  \leq T^{d/\alpha -2}  \intr  \frac{|\widehat{\varphi}_T(z)|^2}{|z|^{\alpha}} \dd{z} \rightarrow 0,\quad \text{as }T\rightarrow +\infty.
\]
where the last integral is finite since $d>\alpha$ and $\widehat{\varphi} \in \SD$.
\paragraph{Convergence of $I_3$} % (fold)
\label{par:convergence_of_i_3_}
Using the Palm formula \eqref{eq:palmFormula} we write
\[
	I_3(T) = \intr \ln(1-v_T(x,0,T)) \int_{\mathcal{M}(\Rd)} \frac{{1+\ddp{\ln(1-v_T)}{\mu}- e^{\ddp{\ln(1-v_T) }{\mu}} }}{ \ddp{\ln(1-v_T) }{\mu} } Q_x(\dd{\mu}) \dd{x},
\]
where $Q_x$ is the Palm measure of measure $Q$. Let us recall that $n_T \approx v_T$, we know also that for small $x$ we have $\ln(1-x) \approx -x$ and $\frac{1-e^{x} + x}{x} \approx -\frac{x}{2}$. Hence $I_3$ is well-approximated by 
\[
	I_{31}(T) := -\frac{1}{2} \intr n_T(x) \int_{\mathcal{M}(\Rd)} \ddp{n_T(x)}{\mu} Q_x(\dd{\mu}) \dd{x}.
\]
First, we will prove first the convergence of $I_{31}$ to the same limit as in \eqref{eq:I4conv}. In the second step we will prove that $|I_3(T) - I_{31}(T)| \rightarrow 0$. We apply definition of $n_T$ (recall  \eqref{sol:tv} and \eqref{eq:simpl1})
\[
	I_{31}(T) = -\frac{1}{2}\intr \intc{T} \T{s_1} \varphi_T(x) \chi_T(s_1) \dd{s_1}  \int_{\mathcal{M}(\Rd)} \ddp{ \intc{T}  \T{s_2} \varphi_T \chi_T(s_2) \dd{s_2} }{\mu}   Q_x(\dd{\mu}) \dd{x}.
\]
Using the Fubini theorem we get (all functions are positive)
\[
	I_{31}(T) = -\frac{1}{2} \intc{T} \intc{T} \chi_T(s_1) \chi_T(s_2) \intr \int_{\mathcal{M}(\Rd)}   \T{s_1} \varphi_T(x)  \ddp{  \T{s_2} \varphi_T   }{\mu}   Q_x(\dd{\mu}) \dd{x} \dd{s_2} \dd{s_1}.
\]
Let us substitute $s_1 \rightarrow T s_1 $ $s_2 \rightarrow T s_2 $ and recall \eqref{def:notation-T}. Using them we acquire
\[
	I_{31}(T) = -\frac{1}{2} \frac{T^2}{F_T^2} \intc{1} \intc{1} \chi(s_1) \chi(s_2) \intr  \int_{\mathcal{M}(\Rd)}  \T{T s_1} \varphi(x)   \ddp{ \T{T s_2} \varphi   }{\mu}   Q_x(\dd{\mu}) \dd{x} \dd{s_2} \dd{s_1}. 
\]
Using Lemma \ref{lem:derivation} with $f_1=f_2=\varphi$ and $F^{*2}(z):= \intr \varphi(-z+z_2) \varphi(z_2) \dd{z_2}$ we obtain
\[
	I_{31}(T) = -\frac{1}{2} \frac{T^2}{F_T^2} \intc{1} \intc{1} \chi(s_1) \chi(s_2) \int_{\mathcal{M}(\Rd)} \ddp{\T{s_1+s_2}F^*}{\mu}   Q_0(\dd{\mu}) \dd{s_2} \dd{s_1}.
\]
We also recall that $F_T^2 = T^{3-d/\alpha}$, hence
\[
	I_{31}(T) = -\frac{1}{2}  \intc{1} \intc{1} \chi(s_1) \chi(s_2)  (s_1+s_2)^{1-d/\alpha} (T(s_1+s_2))^{d/\alpha-1}\int_{\mathcal{M}(\Rd)} \ddp{\T{T\rbr{s_1+s_2}} F^{*} }{\mu} Q_0(\dd{\mu}) \dd{s_2} \dd{s_1}.
\]
By (A3) the inner expression converges to $\mathcal{H}(Q)\norm{F^*}{1}{}=\norm{\varphi}{1}{2}$ and (A2) vindicates use of the Lebesgue dominated converge theorem, therefore
\[
	\lim_{T\rightarrow +\infty} I_{31}(T) = -\frac{1}{2} \mathcal{H}(Q)  \norm{\varphi}{1}{2} \intc{1} \intc{1} \chi(s_1) \chi(s_1)  (s_1+s_2)^{1-d/\alpha} \dd{s_2} \dd{s_1}. 
\]
Let us now consider the inner integral and recall \eqref{def:notation-T}
\begin{multline*}
	\intc{1} \intc{1} \chi(s_1) \chi(s_1)  (s_1+s_2)^{1-d/\alpha} \dd{s_2} \dd{s_1} = \intc{1} \intc{1} \int_{s_1}^1 \int_{s_2}^1 \psi(u_1) \psi(u_2) \dd{u_1} \dd{u_2}    (s_1+s_2)^{1-d/\alpha} \dd{s_2} \dd{s_1} \\
	= \intc{1} \intc{1} \intc{1} \intc{1} 1_{[s_1,1]}(u_1) 1_{[s_2,1]}(u_2) \psi(u_1) \psi(u_2)    (s_1+s_2)^{1-d/\alpha} \dd{u_1} \dd{u_2}  \dd{s_2} \dd{s_1}\\
		= \intc{1} \intc{1} \psi(u_1) \psi(u_2) \intc{1} \intc{1} 1_{[s_1,1]}(u_1) 1_{[s_2,1]}(u_2)     (s_1+s_2)^{1-d/\alpha}   \dd{s_2} \dd{s_1} \dd{u_1} \dd{u_2} \\
		= (2-d/\alpha)^{-1} \intc{1} \intc{1} \psi(u_1) \psi(u_2) \intc{1} 1_{[s_1,1]}(u_1)    \sbr{ (s_1+u_2)^{2-d/\alpha}-s_1^{2-d/\alpha} }   \dd{s_1} \dd{u_1} \dd{u_2} \\ = \rbr{(3-d/\alpha)(2-d/\alpha)}^{-1} \intc{1} \intc{1} \psi(u_1) \psi(u_2) \sbr{(u_1 + u_2)^{3-d/\alpha} - u_1^{3-d/\alpha} -u_2^{3-d/\alpha}} \dd{u_1} \dd{u_2}  \\
		= \rbr{(3-d/\alpha)(2-d/\alpha)}^{-1} \intc{1} \intc{1} \psi(u_1) \psi(u_2) c_h(u_1, u_2)  \dd{u_1} \dd{u_2},
\end{multline*} 
where $h=3-d/\alpha$. This is exactly the same as limit \eqref{eq:I4conv}. To conclude we still have to prove that $|I_3(T) - I_{31}(T)| \rightarrow 0$. This will be done in a few steps in which we will gradually change formula from $I_{3}$ toward $I_{31}$ proving each time that they are asymptotically the same. Let us denote 
\[
	I_{32}(T) := \intr \ln(1-v_T(x)) \int_{\mathcal{M}(\Rd)}  -\frac{1}{2} \ddp{\ln(1-v_T(x))}{\mu} Q_x(\dd{\mu}) \dd{x}. 
\]
Therefore we may estimate
\begin{equation}
	 I_{3}(T)-I_{32}(T) = \intr \ln(1-v_T(x,0,T)) \int_{\mathcal{M}(\Rd)} Z_T  Q_x(\dd{\mu})\dd{x}, \label{eq:z32Z3} 
\end{equation}
where
\[
	Z_T := \frac{{1+\ddp{\ln(1-v_T)}{\mu}- e^{\ddp{\ln(1-v_T(\cdot, 0, T)) }{\mu}} }}{ \ddp{\ln(1-v_T) }{\mu} } + \frac{1}{2}\ddp{\ln(1-v_T)}{\mu}.
\]
Let us notice that \eqref{ineq:tv>v} and Taylor expansion imply that $|Z_T| \cleq \min(\ddp{n_T}{\mu},\ddp{n_T}{\mu}^2) $, therefore by Lemma \ref{lem:convergence} we know that $\ddp{n_T}{\mu} \rightarrow 0$ in probability under $Q_x$ uniformly in $x$. By the diagonal procedure we may construct $H_T \rightarrow 0$ such that $\sup_{x} Q_x(\ddp{n_T}{\mu} >H_T) \rightarrow 0$. We now write
\begin{multline*}
	|I_{32}(T) - I_{3}(T)| = \intr \ln(1-v_T(x)) \int_{\mathcal{M}(\Rd)} Z_T 1_{\cbr{\ddp{n_T}{\mu} \leq H_T}}  Q_x(\dd{\mu})\dd{x} +\\
		\intr \ln(1-v_T(x)) \int_{\mathcal{M}(\Rd)} Z_T 1_{\cbr{\ddp{n_T}{\mu} > H_T}}  Q_x(\dd{\mu})\dd{x}. 
\end{multline*}
We have $|\ln(1-x)| \cleq  |x|$. Using this the above considerations, \eqref{ineq:tv>v} and the H\"older inequality we get
\begin{multline*}
	|I_{32}(T) - I_{3}(T)| \cleq H_T \intr \ln(1-v_T(x)) \int_{\mathcal{M}(\Rd)} \ddp{n_T}{\mu}  Q_x(\dd{\mu})\dd{x} \\
		+\sup_x Q_x(Z_T >H_T)^{\epsilon/(1+\epsilon)} \intr \ln(1-v_T(x)) \cbr{\int_{\mathcal{M}(\Rd)} Z_T^{1+\epsilon}   Q_x(\dd{\mu})}^{1/(1+\epsilon)}   \dd{x}. 
\end{multline*} 
where $\epsilon$ is the same as in assumption (A4). Applying (A4) and $|\ln(1-v_T)| \leq n_T$ we arrive at 
\[
	|I_{32}(T) - I_{3}(T)| \cleq H_T I_{31}(T)  +
		\sup_x Q_x(Z_T >H_T)^{\epsilon/(1+\epsilon)} I_{31}(T) \rightarrow 0.
\]
% Using $|\ln(1-v_T)| \leq n_T$ again, this time in the second summand we get
% \[
% 	|I_{32}(T) - I_{3}(T)| \cleq H_T I_{31}(T) + \sup_x  Q_x(Z_T >H_T)^{\epsilon/(1+\epsilon)} I_{31}(T) \rightarrow 0.
% \]
Next we define
\[
	I_{33}(T) := \frac{1}{2} \intr v_T(x) \int_{\mathcal{M}(\Rd)} \ddp{\ln(1-v_T)}{\mu}  Q_x(\dd{\mu}) \dd{x}. 
\]
Using the inequality $|\ln(1-x) + x| \cleq x^2$ valid for $x\in (0,1/2)$ in conjunction with \eqref{ineq:tv>v} we get
\[
	|I_{33}(T) - I_{32}(T)| \cleq \intr v_T(x)^2 \int_{\mathcal{M}(\Rd)} \ddp{\ln(1-v_T)}{\mu}  Q_x(\dd{\mu}) \dd{x}.
\]
Applying \eqref{ineq:tv>v} and calculating as before we have
\[
	|I_{33}(T) - I_{32}(T)| \cleq \frac{I_{31}(T)}{F_T} \rightarrow 0.
\]
We define also
\[
		I_{34}(T) := -\frac{1}{2} \intr v_T(x) \int_{\mathcal{M}(\Rd)} \ddp{v_T}{\mu} Q_x(\dd{\mu}) \dd{x}.
\]
By \eqref{ineq:tv>v} and  the Taylor formula we have $|\log(1-v_T(x))+v_T(x)|\leq v_T(x)^2 \leq F_T^{-1}n_T(x)$. Now analogously as before we prove  
\[
	|I_{33}(T) - I_{34}(T)| \rightarrow 0.
\]
Further we put
\[
	I_{35}(T) := -\frac{1}{2} \intr n_T(x) \int_{\mathcal{M}(\Rd)} \ddp{v_T}{\mu} Q_x(\dd{\mu}) \dd{x}.
\]
Using inequalities \eqref{eq:n-v} and \eqref{ineq:tv>v} it is easy to get $|I_{34}(T) - I_{35}(T)| \cleq I_{351}(T) + I_{352}(T)$ where
\[
	I_{351}(T) =  \intr \frac{1}{F_T} n_T(x)  \int_{\mathcal{M}(\Rd)} \ddp{v_T}{\mu} Q_x(\dd{\mu}) \dd{x} \leq \frac{1}{F_T} I_{31}(T) \rightarrow 0.
\]
and in the second term is (we use \eqref{ineq:tv>v} under integral $\int_{\mathcal{M}(\Rd)}$)
\[
	I_{352}(T) \cleq  \intr \intc{T} \T{T-s} n_T^2(x,T-s,s) \dd{s}  \int_{\mathcal{M}(\Rd)} \ddp{n_T}{\mu} Q_x(\dd{\mu}) \dd{x}.
\]
It is easy to check that $n_T(x,T-s,s) \cleq n_T(x)$. Using this, \eqref{sol:tv} and the Fubini theorem we write
\[
	I_{352}(T) \cleq   \intc{T}\intc{T} \intr\int_{\mathcal{M}(\Rd)} \T{T-s_1}  n_T^2(x) \ddp{\T{s_2}\varphi_T}{\mu} Q_x(\dd{\mu}) \dd{x} \dd{s_1} \dd{s_2} .
\]
Using Lemma \ref{lem:derivation} with  $f_1 = n_T^2, f_2 = \varphi_T$ and $F^*_T(z) =\intr n_T^2(-z+z_2)\varphi_T(z_2) \dd{z_2}$ we obtain
\begin{equation}
	I_{352}(T) \cleq   \intc{T}\intc{T} \int_{\mathcal{M}(\Rd)} \ddp{\T{T-s_1+s_2}F^*_T}{\mu} Q_0(\dd{\mu}) \dd{s_1} \dd{s_2}. \label{eq:tmp18}
\end{equation}
By Lemma \ref{lem:a2reformulations} we have
\[
	I_{352}(T) \cleq \intc T\intc T(T-s+u)^{1-d/\alpha}\norm{F^*_{T}}{1}{}\dd{s}\dd{s_2}Q_{0}(\dd{\mu}).
\]
Using simple calculations and \eqref{eq:fourier.} we get 
\begin{multline*}
\norm{F^*_{T}}{1}{}=\int_{\R^{d}}n_{T}(x)^{2}\dd{x}\int_{\R^{d}}\varphi_{T}(x)\dd{ x} \ceq F_{T}^{-3}\int_{\R^{d}}\intc{ T}\intc{T}\T{u_{1}}\varphi(x)\T{u_{2}}\varphi(x)\dd{u_{1}}\dd{ u_{2}}\dd{x} 
\\=F_{T}^{-3}\intr|\hat{\varphi}(z)|^{2}\rbr{\frac{1}{|z|^{\alpha}}\rbr{1-e^{-T|z|^{\alpha}}}}^{2}\dd z.
\end{multline*}
Let  \eqref{eq:ineq-exp} with $\delta=1-\frac{1}{2}\frac{d}{\alpha}+\epsilon$. One easily checks that for $\epsilon$ small enough we have $\delta\in (0,1)$. Hence by \eqref{eq:ineq-exp} follows
\[
	\norm{F^*_{T}}{1}{}\cleq F_{T}^{-3}T^{(2-d/\alpha)+2\epsilon}\intr\frac{|\hat{\varphi}(z)|^{2}}{|z|^{d-2\alpha \epsilon }}\dd z=T^{(d/\alpha-5)/2+2\epsilon}.
\]
Further
\begin{multline*}
	I_{352}(T) \cleq T^{(d/\alpha-5)/2+2\epsilon} \intc T\intc T(T-s_1+s_2)^{1-d/\alpha}\dd{s_1}\dd{s_2}Q_{0}(\dd{\mu})\\
	\leq T^{(d/\alpha-5)/2-2\epsilon}T^{3-d/\alpha} = T^{(1-d/\alpha)/2 +2\epsilon} \rightarrow 0,
\end{multline*}
if only $\epsilon$ is small enough. Therefore
\[
	|I_{34}(T) - I_{35}(T)| \rightarrow 0.
\]
Finally, let us notice that by inequality \eqref{eq:n-v} we have $|I_{35}(T) - I_{31}(T)| \cleq I_{361}(T) + I_{362}(T)$ where
\[
	I_{361}(T) = \intr n_T(x) \int_{\mathcal{M}(\Rd)} \ddp{\frac{1}{F_T} n_T(x)}{\mu} Q_x(\dd{\mu}) \dd{x} = 2 \frac{I_{31}(T)}{F_T} \rightarrow 0,
\]
and 
\[
	I_{362}(T) =  \intr n_T(x) \int_{\mathcal{M}(\Rd)} \ddp{\intc{T}\T{T-s} n^2_T(\cdot,T-s,s) \dd{s} }{\mu} Q_x(\dd{\mu}) \dd{x}.
\]
Using the Fubini theorem, inequality $n^2_T(x,T-s,s)\leq n^2_T(x)$ and \eqref{sol:tv} we get 
\[
	I_{362}(T) \leq \intc{T}\intc{T} \intr  \int_{\mathcal{M}(\Rd)} \T{s_1}\varphi_T(x) \ddp{\T{T-s_2} n^2_T  }{\mu} Q_x(\dd{\mu}) \dd{x} \dd{s_1} \dd{s_2}.
\]
Using Lemma \ref{lem:derivation} with $f_1 = \varphi_T, f_2 = n_T^2$ and $F^*_T(z) =\intr n_T^2(-z+z_2)\varphi_T(z_2) \dd{z_2}$ we obtain
\[
	I_{362}(T) \cleq   \intc{T}\intc{T} \int_{\mathcal{M}(\Rd)} \ddp{\T{T-s_2+s_1}F^*_T}{\mu} Q_0(\dd{\mu}) \dd{s_1} \dd{s_2}.
\]
This is the same as \eqref{eq:tmp18}. Therefore $I_{362}(T) \rightarrow 0$ and the proof is concluded.

\subsection{Tightness} % (fold) ---------------------------------------------------subsection--------------------------------------------------------------
\label{sec:tightness}
In the whole section we fix $0\leq t_1<t_2\leq 1$. Our aim is to compute the second moment of the increment, viz. $\ev{}(\ddp{X_T(t_2)}{\varphi} - \ddp{X_T(t_1)}{\varphi})^2 $ and proof \eqref{ineq:tightness}. In this section we use the following definition of $n_T$
\[
	n_T(x) = \int_{Tt_1}^{T t_2} \T{u} \varphi_T(x) \dd{u}.
\]
Obviously we have
\begin{multline}
	L(t_1,t_2,T) := \ev{}  \rbr{\ddp{X_T(t_2)}{\varphi} - \ddp{X_T(t_1)}{\varphi}}^2 = \ev{}  \rbr{ \int_{Tt_1}^{Tt_2} \ddp{N_{u}-\lambda}{\varphi_T} \dd{u}  }^2 \\= 
	\ev{}  \rbr{ \int_{T t_1}^{Tt_2} \ddp{N_{u}}{\varphi_T} \dd{u}  - (Tt_2- Tt_1)\ddp{\lambda}{\varphi_T} }^2 \\
	= \ev{}  \rbr{ \int_{Tt_1}^{Tt_2} \ddp{N_{u}}{\varphi_T}\dd{u}}^2   - 2 (Tt_2- Tt_1)\ddp{\lambda}{\varphi_T} \ev{}  \rbr{ \int_{Tt_1}^{Tt_2} \ddp{N_{u}}{\varphi_T}} \dd{u}  +\rbr{(Tt_2- Tt_1)\ddp{\lambda}{\varphi_T}}^2. \label{eq:l1}
\end{multline}
It is easy to prove that $\lambda$ is the intensity measure of $N_u$ (which follows by the fact that $\lambda$ is the invariant measure of the $\alpha$-stable motion). Using the Fubini theorem and equation \eqref{eq:intensity}  we obtain $\ev{}  \rbr{ \int_{Tt_1}^{Tt_2} \ddp{N_{u}}{\varphi_T} \dd{u} } =    \int_{Tt_1}^{Tt_2} \ev{} \ddp{N_{u}}{\varphi_T} \dd{u} = (Tt_2 - Tt_1) \ddp{\lambda}{\varphi_T}$. Therefore we have
\begin{equation}
	L(t_1,t_2,T) := \ev{}  \rbr{ \int_{Tt_1}^{Tt_2} \ddp{N_{u}}{\varphi_T}\dd{u} }^2   -\rbr{(Tt- Tt)\ddp{\lambda}{\varphi_T}}^2. \label{eq:defL}
\end{equation}
Let us deal with the first term; we first condition with respect to the starting measure $M$
\begin{multline*}
	L_1(t_1,t_2,T) :=   \ev{}  \rbr{ \int_{Tt_1}^{Tt_2} \ddp{N_{u}}{\varphi_T}\dd{u} }^2 = \ev{}  \ev{} \rbr{\left. \rbr{ \int_{Tt_1}^{Tt_2} \ddp{N_{u}}{\varphi_T}\dd{u} }^2\right| M} \\= \ev{}  \ev{} \rbr{\left. \rbr{ \sum_{x\in [M]}\int_{Tt_1}^{Tt_2} \ddp{N^x_{u}}{\varphi_T}\dd{u} }^2\right| M},
\end{multline*}
where by $N^x$ we denote a system starting from $\delta_x$. Let us calculate (slightly abusing notation we use $\ev{}$ also to $\ev{}(\cdot|M)$)
\begin{multline*}
	L_1(t_1,t_2,T) =  \ev{} \sum_{x\in [M]} \ev{} \rbr{ \int_{Tt_1}^{Tt_2} \ddp{N^x_{u}}{\varphi_T}\dd{u} }^2 \\
	+ \ev{} \sum_{x\in [M],y\in [M],\: x\neq y} \ev{} \rbr{ \int_{Tt_1}^{Tt_2} \ddp{N^x_{u}}{\varphi_T}\dd{u} }\rbr{ \int_{Tt_1}^{Tt_2} \ddp{N^y_{u}}{\varphi_T}\dd{u} }.
\end{multline*}
The systems $N^x$ and $N^y$ are independent as long as $x\neq y$, therefore
\[
	\ev{} \rbr{ \int_{Tt_1}^{Tt_2} \ddp{N^x_{u}}{\varphi_T}\dd{u} }\rbr{ \int_{Tt_1}^{Tt_2} \ddp{N^y_{u}}{\varphi_T}\dd{u} }= \ev{} \rbr{ \int_{Tt_1}^{Tt_2} \ddp{N^x_{u}}{\varphi_T}\dd{u} }\ev{} \rbr{ \int_{Tt_1}^{Tt_2} \ddp{N^y_{u}}{\varphi_T}\dd{u} } = (*).
\]
We know that $\ev{} \ddp{N^y_{u}}{\varphi_T} = \T{u}\varphi_T(y)$ (this is standard and can be found e.g. in \cite{Gorostiza:1991aa}). Due to it
\[
	(*) = \rbr{\int_{Tt_1}^{Tt_2} \T{u}\varphi_T(x) \dd{u} } \rbr{\int_{Tt_1}^{Tt_2} \T{u}\varphi_T(y) \dd{u} } = n_T(x) n_T(y)
\]
Now the second summand of $L_1$ can be recalculated as (by assumption (A1) this expression is finite)
\[
	 \ev{} \sum_{x\in [M],y\in [M]} n_T(x) n_T(y) - \sum_{x\in [M]} \rbr{n_T(x) }^2 = \int_{\mathcal{M}(\Rd)} \ddp{\mu}{n_T(x)}^2 - \ddp{\mu}{n_T^2(x)} M(\dd{\mu}).
\]
By Lemma \ref{lem:ippdSecondMoment} and equation \eqref{eq:intensity} this is equal to 
\[
	\int_{\mathcal{M}(\Rd)} \ddp{\mu}{n_T(x)}^2  Q(\dd{\mu}) + \rbr{\intr n_T(x) \dd{x}  }^2- \intr n_T^2(x) \dd{x}.
\]
We also use \eqref{eq:secondmome} with $\Psi(x,s) = \varphi_T(x) 1_{[T t_1, T t_2]}(s)$ and \eqref{eq:intensity} again in order to get
\[
 L_1(t_1,t_2,T) := \intr -v''(0)(x,0,T) \dd{x} +\int_{\mathcal{M}(\Rd)} \ddp{\mu}{n_T(x)}^2  Q(\dd{\mu}) + \rbr{\intr n_T(x) \dd{x}  }^2- \intr n_T^2(x) \dd{x}. 
\]
Let us notice that $\intr n_T(x) \dd{x} = (Tt_1 - Tt_2) \ddp{\lambda}{\varphi_T}$, hence $L$ given by \eqref{eq:defL} writes as
\begin{equation}
	L(t_1,t_2,T) := \underbrace{\int_{\mathcal{M}(\Rd)} \ddp{n_T}{\mu}^2  Q(\dd{\mu})}_{J_1(T)} - \underbrace{\intr n_T^2(x) \dd{x}}_{J_2(T)} + \underbrace{\intr -v''(0)(x,0,T) \dd{x}}_{J_3(T)}. \label{eq:l2}
\end{equation}
We start with a simple case of $J_2$. We have
\[
	J_2(T) = \intr \rbr{\int_{T t_1}^{T t_2} \T{s}\varphi_T(x) \dd{s}  }^2 \dd{x} = \int_{T t_1}^{T t_2} \int_{T t_1}^{T t_2}  \intr  \T{s_1}\varphi_T(x) \T{s_2}\varphi_T(x) \dd{x} \dd{s_1} \dd{s_2}.    
\]
Changing variables $s_1 \rightarrow T s_1 $ and $s_2 \rightarrow T s_2$ and using \eqref{def:notation-T} we get
\[
	J_2(T) = T^2 F_T^{-2} \int_{t_1}^{t_2}\int_{t_1}^{t_2}  \intr  \T{Ts_1}\varphi(x) \T{s_2}\varphi(x) \dd{x} \dd{s_1} \dd{s_2}.
\]
Applying the Fourier transform \eqref{eq:fourier.} we obtain 
\[
	J_2(T) =  T^2 F_T^{-2} \int_{t_1}^{t_2}\int_{t_1}^{t_2} \intr  e^{-Ts_1|z|^\alpha} e^{-Ts_2|z|^\alpha} |\widehat{\varphi}(z)|^2 \dd{z} \dd{s_1} \dd{s_2}.
\]
So
% \[
% 	J_2(T) \leq T^2 F_T^{-2} \intr |\widehat{\varphi}(z)|^2 \int_{t_1}^{t_2}\int_{t_1}^{t_2}     e^{-Ts_1|z|^\alpha} e^{-Ts_2|z|^\alpha}   \dd{s_1} \dd{s_2} \dd{z},
% \]
\[
	J_2(T) \leq T^2 F_T^{-2} \intr |\widehat{\varphi}(z)|^2  \rbr{\frac{1}{T|z|^\alpha}\rbr{e^{-Tt_2|z|^\alpha} - e^{-Tt_1|z|^\alpha}} }^2  \dd{z}.
\]
Further
% \[
% 	J_2(T) \leq T^2 F_T^{-2} \intr |\widehat{\varphi}(z)|^2  \rbr{\frac{ e^{-Tt_1|z|^\alpha}}{T|z|^\alpha}\rbr{e^{-T(t_2-t_1)|z|^\alpha} - 1} }^2  \dd{z},
% \]
\[
	J_2(T) \leq T^2 F_T^{-2} \intr |\widehat{\varphi}(z)|^2  \rbr{\frac{1}{T|z|^\alpha}\rbr{e^{-T(t-s)|z|^\alpha} - 1} }^2  \dd{z}.
\]
We use inequality \eqref{eq:ineq-exp} with $\delta = (3-d/\alpha)/2$
\[
	J_2(T) \leq T^2 F_T^{-2} \intr |\widehat{\varphi}(z)|^2  \rbr{\frac{1}{T|z|^\alpha}\rbr{ T (t-s) |z|^\alpha}^{(3-d/\alpha)/2} }^2  \dd{z},
\]
\[
	J_2(T) \leq (t-s)^{3-d/\alpha} \intr \frac{|\widehat{\varphi}(z)|^2}{|z|^{d-\alpha}}   \dd{z} \cleq  (t-s)^{3-d/\alpha}.
\]
Now we deal with  $J_1$. Using the Palm formula \eqref{eq:palmFormula} we obtain
\[
	J_1(T) = \intr n_T(x) \int_{\mathcal{M}(\Rd)} \ddp{ n_T }{\mu}  Q_x(\dd{\mu}).
\]
Using the definition of $n_T$ we write
\[
	J_1(T) = \intr \int_{T t_1}^{T t_2} \T{s_1} \varphi_T(x)  \dd{s_1}  \int_{\mathcal{M}(\Rd)} \ddp{\int_{T t_1}^{T t_2} \T{s_2} \varphi_T(x)  \dd{s_2} }{\mu}  Q_x(\dd{\mu}) \dd{x}.
\]
Now we apply the Fubini theorem and recall \eqref{def:notation-T} to get
\[
	J_1(T) = F_T^{-2} \int_{T t_1}^{T t_2} \int_{T t_1}^{T t_2} \intr  \T{s_1} \varphi(x) \int_{\mathcal{M}(\Rd)} \ddp{\T{s_2} \varphi(x)}{\mu}  Q_x(\dd{\mu}) \dd{x} \dd{s_1} \dd{s_2}.
\]
Using Lemma \ref{lem:derivation}  with $f_1=\varphi, f_2=\varphi$ and $F^*(z_1)=\int_{\Rd}\varphi(-z_1+z_2) \varphi(z_2) \dd{z}_2 $ we get
\begin{equation}
	J_1(T) = F_T^{-2} \int_{T t_1}^{T t_2} \int_{T t_1}^{T t_2}   \int_{\mathcal{M}(\Rd)} \ddp{\T{s_1+s_2} F^*(x)}{\mu}  Q_0(\dd{\mu}) \dd{s_1} \dd{s_2}. \label{eq:l3}
\end{equation}
Now by Lemma \ref{lem:a2reformulations} we have (we recall that $h=(3-d/\alpha)$)
\begin{multline*}
	J_1(T) \cleq F_T^{-2} \int_{T t_1}^{T t_2} \int_{T t_1}^{T t_2} (s_1+s_2)^{1-d/\alpha} \dd{s_1} \dd{s_2} \\
	=F_T^{-2}\rbr{(Tt_2 +T t_2)^{h} - (Tt_1 +T t_2)^{h} -(Tt_2 +T t_1)^{h} +(Tt_1 +T t_1)^{h}}\\
	=(2 t_2)^{h} - 2(t_1 +t_2)^{h} +(2 t_1)^{h} \cleq (t_2-t_1)^2.
\end{multline*}
Finally, we deal with $J_3$. Using equation \eqref{eq:vpp} we write
\begin{multline*}
	J_3(T) = 2 \intr  \int_{0}^{T}\mathcal{T}_{T-s}\left[  \Psi(\cdot,T-s) n_\Psi(\cdot, T-s,s) + V n_\Psi(\cdot, T-s,s)^2  \right] \dd{s} \dd{x} \\
	\ceq  \underbrace{\intr  \int_{0}^{T} \Psi(x,T-s) n_{\Psi}(x, T-s,s)  \dd{s} \dd{x}}_{J_{31}(T)} + \underbrace{\intr  \int_{0}^{T}  n_{\Psi}(x, T-s,s)^2   \dd{s} \dd{x}}_{J_{32}(T)}.
\end{multline*}
We start with $J_{31}$, let us recall that $\Psi(x,s) = \varphi_T(x) 1_{[T t_1, T t_2]}(s)$), substitute $s\rightarrow sT$ and $u\rightarrow Tu$  and apply \eqref{def:notation-T}
\begin{multline*}
J_{31}(T) = \intr  \int_{0}^{T} \int_{0}^{s} \varphi_T(x) 1_{[T t_1, T t_2]}(s) \T{s-u} \varphi_T(x) 1_{[T t_1, T t_2]}(T-u) \dd{u}   \dd{s} \dd{x} =\\
 	\frac{T^2}{F_T^2} \intr  \int_{0}^{1} \int_{0}^{s} \varphi(x) 1_{[t_1,t_2]}(s) \T{T(s-u)} \varphi(x) 1_{[t_1, t_2]}(1-u) \dd{u}   \dd{s} \dd{x}.
\end{multline*}
Applying the Fourier transform \eqref{eq:fourier.} we get
\[
	J_{31}(T) = \frac{T^2}{F_T^2} \intr  \int_{0}^{1} \int_{0}^{s} |\widehat{\varphi}(z)|^2 1_{[t_1, t_2]}(s) e^{-T(s-u)|z|^\alpha} 1_{[t_1, t_2]}(1-u) \dd{u}   \dd{s} \dd{z}.
\]
We may easily estimate
\begin{multline*}
	J_{31}(T) \cleq \frac{T}{F_T^2} \intr  \int_{0}^{1}  \frac{|\widehat{\varphi}(z)|^2}{|z|^\alpha} 1_{[t_1, t_2]}(s) e^{-T(1 - t_2)|z|^\alpha} \rbr{1 - e^{-T(t_2-t_1)|z|^\alpha}}    \dd{s} \dd{z} \\ \cleq (t_2-t_1) \frac{T}{F_T^2} \intr \frac{|\widehat{\varphi}(z)|^2}{|z|^\alpha} \rbr{1 - e^{-T(t_2-t_1)|z|^\alpha}}  \dd{z},
\end{multline*}
we use inequality \eqref{eq:ineq-exp} in order to get
\[
		J_{31}(T) \cleq (t_2-t_1)^{1+\delta} \frac{T}{F_T^2} T^{\delta} \intr \frac{|\widehat{\varphi}(z)|^2}{|z|^{\alpha-\delta}}\dd{z} \cleq (t_2-t_1)^{1+\delta} 
\]
as one can easily find a proper $\delta>0$. Let us now proceed to $J_{32}$. Using \eqref{sol:tv} and changing variables $s\rightarrow Ts$, $u\rightarrow Tu$
\begin{multline*}
J_{32}(T) = \intr  \int_{0}^{T} \rbr{\intc{s} \T{s-u} \varphi_T(x) 1_{[T t_1, T t_2]}(T-u) \dd{u}  }^2    \dd{s} \dd{x} \\\ceq \frac{T^3}{F_T^2} \intr  \int_{0}^{1} \rbr{\intc{s} \T{T(s-u)} \varphi(x) 1_{[t_1, t_2]}(1-u) \dd{u}  }^2    \dd{s} \dd{x}.
\end{multline*}
Using equation \eqref{eq:fourier.} we obtain
\[
	J_{32}(T) \ceq \frac{T^3}{F_T^2} \intr  \int_{0}^{1} |\widehat{\varphi}(z)|^2 \rbr{\intc{s} e^{-T(s-u)|z|^\alpha}1_{[t_1, t_2]}(1-u)  \dd{u}  }^2    \dd{s} \dd{x}.
\]
Next, we write
\[
	J_{32}(T) \cleq \frac{T^3}{F_T^2} \int_{0}^{1}\intc{s} \intc{s} \intr   e^{-T(2s-u_1-u_2)|z|^\alpha} 1_{[t_1, t_2]}(1-u_1) 1_{[t_1, t_2]}(1-u_2)  \dd{u_1}\dd{u_2} \dd{s} \dd{z}.
\]
Integrating with respect to $z$ yields
% \[
% 	J_{32}(T) \cleq \int_{0}^{1}\intc{s} \intc{s} ((2s - u_1 - u_2))^{-d/\alpha} \intr  \norm{\widehat{\varphi}}{\infty}{2} e^{-|z|^\alpha} \chi(1-u_1)\chi(1-u_2) \dd{u_1}\dd{u_2} \dd{s} \dd{z},
% \]
\[
	J_{32}(T) \cleq \int_{0}^{1}\intc{s} \intc{s} ((2s - u_1 - u_2))^{-d/\alpha} 1_{[t_1, t_2]}(1-u_1) 1_{[t_1, t_2]}(1-u_2)  \dd{u_1}\dd{u_2} \dd{s},
\]
Further
\begin{multline*}
	J_{32}(T) \cleq \int_{1-t_2}^{1-t_1} \int_{1-t_2}^{u_1} \int_{u_1}^{1} (2s - u_1 - u_2)^{-d/\alpha}   \dd{s} \dd{u_1}\dd{u_2} \ceq \\	 \int_{1-t_2}^{1-t_1} \int_{1-t_2}^{u_1}  (2 - u_1 - u_2)^{-d/\alpha+1}   \dd{u_2}\dd{u_1} -	\int_{1-t_2}^{1-t_1} \int_{1-t_2}^{u_1}  ( u_1 - u_2)^{-d/\alpha+1}  \dd{u_2}\dd{u_1}.
\end{multline*}
Let us denote $\delta = t_2 - t_1$ and write
\[
	f(\delta) := \frac{1}{\delta^{3-d/\alpha}} \int_{1-t_2}^{1-t_2+\delta} \int_{1-t_2}^{u_1}  (2 - u_1  - u_2)^{-d/\alpha+1}   \dd{u_1}\dd{u_2}.
\]
It is easy to check that by l'Hôpital's rule we have
\[
	\lim_{\delta\rightarrow 0^+} f(\delta) \ceq \lim_{\delta\rightarrow 0^+} \frac{1}{\delta^{2-d/\alpha}}  \rbr{ \int_{1-t_2}^{1-t_2+\delta}  (1+t_2-\delta  - u_2)^{-d/\alpha+1}   \dd{u_2} } \leq c.
\]
Hence we get  $\int_{t_1}^{t_1+\delta} \int_{t_1}^{u_1}  (2 - u_1  - u_2)^{-d/\alpha+1}   \dd{u_1}\dd{u_2} \cleq (t_2 - t_1)^{3-d/\alpha} $. The other term is estimated by
\[
		\int_{t_1}^{t_2} \int_{t_1}^{u_1}  ( u_1 - u_2)^{1-d/\alpha}  \dd{u_2}\dd{u_1} \ceq \int_{t_1}^{t_2}    ( u_1 - t_1)^{2-d/\alpha} \dd{u_1} \ceq ( t_2 - t_1)^{3-d/\alpha}.
\]

% paragraph convergence_of_i_4_ (end)
\section{Equilibrium measure and its Palm measure} % (fold)
\label{sec:equlibrium}
Let us recall the description of clan decomposition of the equilibrium measure given in Remark \ref{rem:examples}. In this section we will compute the Laplace transform and the first moment of $\xi^x_\infty$, see equation \eqref{eq:clan}. To this end we first consider the following Laplace transform 
\begin{equation}
	w(x,t) = 1-\lap{-\ddp{N_t^x}{f}},	 \label{eq:oneParticleLaplace}
\end{equation}
where $N^x$ is the branching particle system  with starting condition $N_0 = \delta_x$ and $f:\Rd\mapsto \R_+$ is some measurable function. It is known \cite[(2.3)]{Gorostiza:1991aa} that
\begin{equation}
	 w(x,t) = \T{t}(1-e^{-f(\cdot)})(x) - \frac{V}{2}\intc{t} \T{t-s} w(x,s)^2 \dd{s}.\label{eq:oneParticleEquation}
\end{equation}
Using this equation we will now calculate the Laplace transform of $\xi_t$.  Let us denote by $Poiss$ a realisation of $\nu$
\begin{equation}
	H(t,x) := \lap{-\ddp{\xi^x_t}{f}} = \ev{} \ev{ \sbr{\left.  \rbr{\exp\cbr{-\ddp{\xi^x_t}{f}}}\right| Poiss,X^x}}. \label{eq:clanLaplaceDef}
\end{equation}
Further, let $t_1, t_2,\ldots, t_n $ be the times when $Poiss$ ticked ($n$ is an a.s. finite random variable ). We have
\[
	H(t,x) = \ev{} \ev{ \sbr{\left.  \rbr{\exp\cbr{-\ddp{\sum_{i=1}^n \zeta_{t_i}^{t_i,X_{t_i}^x} }{f} }}\right| Poiss,X^x}} = \ev{} \prod_{i=1}^n \ev{ \sbr{\left.  \rbr{\exp\cbr{-\ddp{ \zeta_{t_i}^{t_i,X_{t_i}^x} }{f} }}\right| Poiss,X^x}}.
\] 
Utilising \eqref{eq:oneParticleLaplace} and the definition \eqref{eq:semigroup-density} we arrive at
\[
	H(t,x) = \ev{} \rbr{\prod_{i=1}^n \ev{}\rbr{\left. (1-w(X_{t_i}, t_i)) \right| Poiss, X^x}} = \ev{} \rbr{\prod_{i=1}^n  \T{t_i} (1-w(x, t_i))}.
\]
Using the properties of the Poisson random fields we obtain
\begin{equation}
		H(t,x) = \exp \cbr{ - V \intc{t} \T{s} w(x,s) \dd{s}  }. \label{eq:clan-measure}
\end{equation}
% section section_name (end)
\paragraph{Expectation of $\xi^x_\infty$} % (fold)
\label{sub:moments_of_xi_x_}
We denote $w(x,t,\theta)$ given by \eqref{eq:oneParticleLaplace} defined for $\theta f$. Obviously, it fulfils an analogue of \eqref{eq:oneParticleEquation}. We easily check that
\[
	w'(x,t,\theta) =  \T{t}( f(\cdot)e^{-\theta f(\cdot)})(x) -  V \intc{t} \T{t-s} \sbr{w'(\cdot,s,\theta)w(\cdot,s,\theta)}(x) \dd{s},
\]
\[
	w'(x,t,0) =  \T{t} f(x).
\]
% Further
% \[
% 	w''(x,t,\theta) =  -\T{t}( f(\cdot)^2e^{-\theta f(\cdot)})(x) -  V \intc{t} \T{t-s} \sbr{w''(\cdot,s,\theta)w(\cdot,s,\theta) + w'(\cdot,s,\theta)^2}(x) \dd{s}, 
% \]
% \[
% 	w''(x,t,0) =  -\T{t} f(x)^2 -  V \intc{t} \T{t-s} \sbr{ w'(\cdot,s,0)^2}(x) \dd{s} = -\T{t} f(x)^2 -  V \intc{t} \T{t-s} \sbr{ (\T{s} f(x))^2}(x) \dd{s}.
% \]
We can calculate the moments of $\ddp{\xi_t^x}{f}$ by differentiating $H(x,t,\theta)$ which is an analogue of \eqref{eq:clanLaplaceDef} defined for $\theta f$. In fact we may write
\[
H(t,x,\theta) :=  \exp \cbr{- V {\intc{t} \T{s}w(x,s,\theta) \dd{s}  }}.
\]
Differentiation with respect to $\theta$ yields
\[
	H'(t,x,\theta) =  \rbr{-V{\intc{t} \T{s}w'(x,s,\theta) \dd{s}}}  \exp \cbr{- V{\intc{t} \T{s}w(x,s,\theta) \dd{s}  }}.
\]
which evaluated at $\theta = 0$ is
\[
	H'(t,x,0) = -  V\intc{t} \T{2s} f(x) \dd{s}.
\]
% Similarly we obtain the second derivative 
% \[
% 	H''(t,x,\theta) = \rbr{-{V\intc{t} \T{s}w''(x,s,\theta) \dd{s}} + \rbr{V{\intc{t} \T{s}w'(x,s,\theta) \dd{s}}}^2}  \exp \cbr{- V{\intc{t} \T{s}w(x,s,\theta) \dd{s}  }}.
% \]
% And again we evaluate it at $\theta =0$
% \[
% 	H''(t,x,0) = {-V{\intc{t} \T{s}w''(x,s,0) \dd{s}} + \rbr{V{\intc{t} \T{s}w'(x,s,0) \dd{s}}}^2}. % \exp \cbr{- \sbr{\intc{t} \T{s}w(x,s,\theta) \dd{s}  }}
% \]
Using the properties of the Lebesgue transform we read
\[
	\ev{\ddp{\xi_t^x}{f}} = V \intc{t} \T{2s} f(x) \dd{s}%,\quad \ev{\ddp{\xi_t^x}{f}^2} = -{V\intc{t} \T{s}w''(x,s,0) \dd{s}} + \rbr{V{\intc{t} \T{s}w'(x,s,0) \dd{s}}}^2.
\]
Applying the monotone convergence theorem and using the fact that system is persistent (see \cite{Gorostiza:1991aa}) obtain
\begin{equation}
	\ev{\ddp{\xi_\infty^x}{f}} = \frac{V}{2} \intc{+\infty} \T{s} f(x) \dd{s} =\frac{V c_{\alpha,d}}{2}   \intr \frac{f(y)}{|x-y|^{d-\alpha} } \dd{y}, \label{eq:equlibrium_intensity}
\end{equation}
where $c_{\alpha,d} = \Gamma\rbr{\frac{d-\alpha}{2}} (2^\alpha \pi^{d/2} \Gamma(\alpha/2))^{-1}$ (see \cite[(1.32)]{Bogdan:2009le}).
 % \begin{equation}
 % 	\ev{\ddp{\xi_\infty^x}{f}^2} = \rbr{\ev{\ddp{\xi_\infty^x}{f}}}^2 + 1/2 V \inti \T{s} f(x)^2 \dd{s} + V \inti \intc{s} \T{2s-u} \sbr{ \T{u} f(x)  }^2 \dd{u} \dd{s}. \label{eq:equilibrium_variance}
 % \end{equation}

Now, we would like to calculate $\intr p_1(y) \frac{1}{|y|^{d-\alpha}} \dd{y} $. Let us denote
\[
	F(x) = \intr p_1(y) \frac{1}{|x-y|^{d-\alpha}} \dd{y}. %= c_{\alpha,d}^{-1}\mathcal{G} p_1(x).
\]
Using \eqref{eq:fourier.} and \eqref{eq:equlibrium_intensity} we easily obtain
\[
	\widehat{F}(z) = c_{\alpha,d}^{-1} e^{-|z|^\alpha} \frac{1}{|z|^\alpha}.
\]
Using the inverse transform we conclude
\[
	F(0) = \frac{c_{\alpha,d}^{-1}}{(2\pi)^d} \intr \widehat{F}(z) \dd{z} = \frac{c_{\alpha,d}^{-1}}{(2\pi)^d} \intr e^{-|z|^\alpha} \frac{c_{\alpha,d}^{-1}}{|z|^\alpha} \dd{z} = \frac{2}{(2\pi)^d} \frac{\pi^{d/2}}{\Gamma(d/2)} \inti e^{-r^\alpha} \frac{1}{r^\alpha} r^{d-1} \dd{r}.
\]
Now we substitute $r \rightarrow  s^{1/\alpha}$ getting
\begin{multline}
F(0) = \frac{2 c_{\alpha,d}^{-1} c_{\alpha,d}^{-1}}{(2\pi)^d} \frac{\pi^{d/2}}{\Gamma(d/2)\alpha} \inti e^{-s} \frac{1}{s} s^{(d-1)/\alpha} s^{1/\alpha-1} \dd{s} \\= \frac{2c_{\alpha,d}^{-1}}{(2\pi)^d} \frac{\pi^{d/2}}{\Gamma(d/2)\alpha} \inti e^{-s} s^{d/\alpha -2}  \dd{s} = \frac{2c_{\alpha,d}^{-1}}{(2\pi)^d} \frac{\pi^{d/2}}{\Gamma(d/2)\alpha} \Gamma(d/\alpha -1)	 \label{eq:tmpIntegral}
\end{multline}

\subsection*{Appendix} % (fold)
\label{sub:subsection_name}
\begin{proof}(of Fact \ref{fact:measure})
	Let us first transform \eqref{eq:cond1} using \eqref{eq:self-similar}
	\[
		f(t):=t^{d/\alpha - 1}\ddp{p_t}{\Lambda_0} = t^{d/\alpha - 1} \intr p_t(x) \Lambda_0(\dd{x} ) =   t^{- 1} \intr p_1( t^{-1/\alpha} x) \Lambda_0(\dd{x} ).
	\]
	Let us denote $F(t) = \Lambda_0(B(t))$. The density $p_t$  of the symmetric $\alpha$-stable process is rotationally symmetric hence we may denote $p(|x|):=p(x)$. It is also know that it is non-increasing. We also denote $a_{t,i} =\frac{i}{n} t^{1/\alpha} $. It is straightforward to check that
	\begin{equation}
		t^{-1}\sum_{i=0}^\infty p_{1}\rbr{ (i+1)/n } \rbr{F\rbr{ a_{t,i+1}} - F\rbr{a_{t,i}}} \leq f(t) \leq t^{-1} \sum_{i=0}^\infty p_{1}\rbr{ i/n } \rbr{F\rbr{a_{t,i+1}} - F\rbr{a_{t,i}}}. \label{eq:tmp11}
	\end{equation}
	Denote the right most part of the last expression by $\overline{D}_{t,n}$. It is easy to check that $ t^{-1}  \rbr{F\rbr{a_{t,i+1}} - F\rbr{a_{t,i}}} \rightarrow \mathcal{H}(\Lambda) \rbr{\rbr{\frac{i+1}{n}}^\alpha - \rbr{\frac{i}{n}}^\alpha}$ and $|t^{-1} \rbr{F\rbr{a_{t,i+1}} - F\rbr{a_{t,i}}}| \cleq i^{\alpha - \epsilon}$. Using the dominated Lebesgue convergence theorem we have
	\[
		\overline{D}_{n} := \lim_{t\rightarrow +\infty} \overline{D}_{t,n} =  \mathcal{H}(\Lambda) \sum_{i=0}^\infty p_{1}\rbr{ \frac{i}{n}} \rbr{\rbr{\frac{i+1}{n}}^\alpha - \rbr{\frac{i}{n}}^\alpha  }.
	\]
	We use inequalities  $\frac{\alpha}{n} \rbr{\frac{i}{n}}^{\alpha -1}\leq \rbr{\frac{i+1}{n}}^\alpha - \rbr{\frac{i}{n}}^\alpha \leq \frac{\alpha}{n} \rbr{\frac{i+1}{n}}^{\alpha -1}$ for $\alpha>1$ (for $0<\alpha <1$ one have to reverse inequality signs) in order to prove 
	\[
		\lim_{n\rightarrow +\infty} \overline{D}_n = \mathcal{H}(\Lambda) \alpha \inti p_1(x) x^{\alpha -1} \dd{x}.  
	\]
	Using the same technique we proceed with ``the left inequality'' in \eqref{eq:tmp11}.
\end{proof}

% section conclussion (end)

\subsection*{Acknowledgements} The author would like to thank Prof. T. Bojdecki for introducing him to the problem and Dr. A. Talarczyk for fruitful discussions. 

\bibliographystyle{abbrv} 
\bibliography{bibliography}

\end{document}